\documentclass[12pt]{amsart}
\usepackage{geometry}                
\usepackage{graphicx}
\usepackage{amssymb}
\usepackage{epstopdf}
\usepackage{amsmath}
\usepackage{color}
\usepackage{hyperref}
\usepackage{mathtools}
\usepackage{pdfsync}
 \usepackage{tikz}
 \usepackage{float}

\usepackage[all]{xy}
\xyoption{matrix}
\xyoption{arrow}

\usetikzlibrary{arrows,decorations.pathmorphing,backgrounds,positioning,fit,petri}

 \tikzset{help lines/.style={step=#1cm,very thin, color=gray},
help lines/.default=.5} 
\tikzset{thick grid/.style={step=#1cm,thick, color=gray},
thick grid/.default=1} 

\newtheorem{thm}{Theorem}[section]
\newtheorem{lem}[thm]{Lemma}
\newtheorem{cor}[thm]{Corollary}
\newtheorem{prop}[thm]{Proposition}

\theoremstyle{definition}
\newtheorem{defn}[thm]{Definition}
\newtheorem{eg}[thm]{Example}

\theoremstyle{remark}
\newtheorem{rem}[thm]{Remark}

\numberwithin{equation}{section}

 \newcommand{\symm}[2]{{\dfrac{\color{blue}#1}{\color{red}#2}}}

\newcommand{\noans}[1]{}

\newcommand{\into}{\hookrightarrow}
 \newcommand{\onto}{\twoheadrightarrow}
 \newcommand{\cof}{\rightarrowtail}

 \newcommand{\ot}{\leftarrow}

\DeclareMathOperator{\im}{im}
\DeclareMathOperator{\coker}{coker}
\DeclareMathOperator{\Hom}{Hom}%
\DeclareMathOperator{\End}{End}%
\DeclareMathOperator{\undim}{\underline{dim}}
 
\newcommand{\field}[1]{\mathbb{#1}}
\newcommand{\ZZ}{\ensuremath{{\field{Z}}}}
\newcommand{\CC}{\ensuremath{{\field{C}}}}
\newcommand{\RR}{\ensuremath{{\field{R}}}}
\newcommand{\QQ}{\ensuremath{{\field{Q}}}}

\newcommand{\commentout}[1]{}

\newcommand{\cA}{\ensuremath{{\mathcal{A}}}}

\newcommand{\cB}{\ensuremath{{\mathcal{B}}}}
\newcommand{\cC}{\ensuremath{{\mathcal{C}}}}

\newcommand{\cG}{\ensuremath{{\mathcal{G}}}}
\newcommand{\cH}{\ensuremath{{\mathcal{H}}}}

\newcommand{\cP}{\ensuremath{{\mathcal{P}}}}
\newcommand{\cQ}{\ensuremath{{\mathcal{Q}}}}

\newcommand{\cS}{\ensuremath{{\mathcal{S}}}}

\newcommand{\cU}{\ensuremath{{\mathcal{U}}}}
\newcommand{\cV}{\ensuremath{{\mathcal{V}}}}
\newcommand{\cW}{\ensuremath{{\mathcal{W}}}}

\newcommand{\cX}{\ensuremath{{\mathcal{X}}}}
\newcommand{\cY}{\ensuremath{{\mathcal{Y}}}}

\newcommand\ki[1]{{#1}}

\definecolor{ki}{RGB}{150,50,0}

\definecolor{rm}{RGB}{0,0,200}

\theoremstyle{plain}
\newtheorem{innercustomthm}{Theorem}
\newenvironment{customthm}[1]
    {\renewcommand
    \theinnercustomthm{#1}\innercustomthm}{\endinnercustomthm}

\title{pseudo-torsion classes}
\author{Kiyoshi Igusa}
\address{Department of Mathematics, Brandeis University, Waltham, MA 02454}\email{igusa@brandeis.edu}
\author{RAY MARESCA}
\address{Department of Mathematics, Bowdoin College, 8600 College Station\\ Brunswick, Maine 04011, United States of America}
\email{r.maresca@bowdoin.edu}



\begin{document}

\begin{abstract}
      For a finite dimensional algebra $\Lambda$, we consider a torsion class $\cG$ in $mod$-$\Lambda$, which is not necessarily functorially finite. We construct a wall-and-chamber structure for $\cG$ where the chambers are the connected components of the complement of the union of walls. \ki{We call these ``pseudo-chambers''.} We also consider ``infinitesimal chambers". \ki{To each pseudo-chamber $\cU$  we associate a ``pseudo-torsion class'' $\cP(\cU)$ and a ``pseudo-torsionfree class'' $\cQ(\cU)$ and show that they determine $\cU$. Furthermore, we show that if both $\cP(\cU)$ and $\cQ(\cU)$ are finitely generated, then $\cU$ is an open set.} We consider ``green paths'' in the stability space and associate to them Harder-Narasimhan stratifications of $\cG$. This paper is part of a series of papers whose goal is to study the ``ghosts'' which are remnants of the indecomposable $\Lambda$-modules which do not lie in $\cG$.

In the special case when our torsion class is all of $mod$-$\Lambda$, we are in the classical well-known setting. All of our results apply to this classical setting. The ``pseudo-torsion classes'' are torsion classes. We point out that we do not take the closure of the set of walls. So, we get more pseudo-chambers and our results are new even in this classical setting.
\end{abstract}

\maketitle

\tableofcontents

\section*{Introduction}\label{sec0}

{Given a torsion class $\cG$ in $mod\text-\Lambda$ \cite{Apostolos-Idun} we consider a subcategory of $\cG$ with the same objects but with only ``strict'' morphisms. Strict morphisms, along with other properties, have kernels in $\cG$  and, as with all morphisms in $\cG$, they have image and cokernel in $\cG$. In $\cG$ we define pseudo-torsion classes to be classes of objects of $\cG$ closed under what we call {``strict''} extensions, given by short exact sequences where both morphisms are strict, and having all strict quotients of its objects. Pseudo-torsionfree classes are defined similarly. They are closed under strict extension and contain all strict subobjects of its objects. We also consider pseudo-wide subcategories of $\cG$. These are the full subcategories of $\cG$ with its strict morphisms which are closed under strict extension and contain the kernel, image and cokernel of every strict morphism between its objects. When $\cG=mod\text-\Lambda$ these notions correspond to the usual notions of torsion classes, torsionfree classes and wide subcategories since all morphisms are strict in that case.}

We consider the stability diagram for the torsion class $\cG$ in which we take only the walls labeled with objects of $\cG$. These diagrams appeared in the first author's research in algebraic $K$-theory and in this paper, we provide a representation theoretic explanation. Some walls in this diagram appear as they do in the stability diagram for $mod\text-\Lambda$; however, other walls become larger due to our strict subobject and strict quotient conditions. More precisely, the stability space is $V=\Hom(K_0\Lambda,\RR)\cong \RR^n$ where $n$ is the number of simple $\Lambda$-modules. Every finitely generated module $M$ has dimension vector $[M]\in K_0\Lambda$ and, for every $\theta\in V$, we have $\theta([M])\in\RR$ which we denote simply as $\theta(M)$. 

For every $M\in\cG$ we have the ``pseudo-wall'' $D_\cG(M)$ labeled with $M$. This is the set of all $\theta\in V$ so that $\theta(M)=0$ and $\theta(M')\le0$ for all ``strict subobjects'' $M'$ of $M$. We define the ``pseudo-chambers'' of $\cG$ to be the connected components of the complement of the union of these pseudo-walls. One of our main results is that each pseudo-chamber has an associated pseudo-torsion class and these pseudo-torsion classes are distinct. In the classical case \cite{BST}, the chambers are the components of the complement of the closure of the set of walls. This misses all the ``thin'' chambers which we also consider.

We define a ``green path'' in $V$ to be a smooth path $\theta_t$ so that, whenever this path goes through a wall, say $\theta_{t_0}\in D_\cG(M)$ it does so in the positive direction: 
\[
\frac{d}{dt}(\theta_t(M))|_{t=t_0}>0.
\]
 We also require that, for $t<<0$, $\theta_t$ should start in the negative octant where $\theta_t(M)<0$ and, for $t>>0$, we should have $\theta(M)>0$ for all nonzero $M\in mod\text-\Lambda$. This is a variation of Bridgeland's construction \cite{Bridgeland} which, together with \cite{King}, is the origin of the study of stability diagrams.

In the main part of the paper, we show that green paths give ``strict'' Harder-Narasimhan (HN) stratifications of $\cG$. We also obtain maximal ``strict'' forward Hom-orthogonal (FHO) sequences.

This process works exceptionally well when there are only ``thin'' walls. The situation becomes considerably more unwieldy when there are thick walls. An example of a thick wall is the ``null wall'' in the tame hereditary case. In our next paper \cite{IM2} we will deal with thick walls, which we call ``Level 1'' walls by ``refining'' them. Details will be given in the next paper, but we outline the first steps of the refinement process in the last section. We observe that $mod\text-\Lambda$ is a torsion class. So, our constructions apply to this classical case.

Section \ref{sec1} gives the basic definitions of strict morphisms, {\color{black} strict extensions}, pseudo-torsion classes, pseudo-torsionfree classes, and pseudo-wide subcategories. We consider a torsion class $\cG$ of $mod\text-\Lambda$ with only strict morphisms. This category, although not abelian, behaves quite well.
{\color{black}
\begin{customthm}{A}[Theorem \ref{thm: strict torsion class is a category} and Proposition \ref{prop: ker and coker in G}]
    For any torsion class $\cG \subseteq mod\text{-}\Lambda$, the class consisting of the objects of $\cG$ but only strict morphisms is a category with kernels, cokernels, and images.
\end{customthm} }

Analogous to the classical case \cite{Apostolos-Idun}, we prove the basic theorem that pseudo-torsion classes and pseudo-torsionfree classes are dual in the sense that they are perpendicular under strict morphisms, i.e., there are no nonzero strict morphisms from a pseudo-torsion class to the corresponding pseudo-torsionfree class.
{\color{black}
\begin{customthm}{B}[Theorem \ref{thm: duality between P,Q} and Theorem \ref{thm: tM and fM}]
    In a torsion class $\cG \subseteq mod\text{-}\Lambda$, there is a bijection between pseudo-torsion classes and pseudo-torsionfree classes (i.e. for any pseudo-torsion class $\cP \subseteq \cG$, we have    $\,^\perp(\cP^\perp)=\cP$). Moreover, given any pseudo-torsion pair $(\cP,\cQ)$ and module $M \in \cG$, there is a unique strict extension $tM \cof M \onto fM$ with $tM \in \cP$ and $fM \in \cQ$.
\end{customthm} }

Section \ref{new sec2} constructs the wall-and-chamber structure in the stability diagram of our torsion class $\cG$. {\color{black} We consider two stability diagrams: the reduced stability diagram in which we only consider modules in $\cG$ that are $\Lambda$-indecomposable, and the full stability diagram. In the latter, even in type $A$, there may be infinitely many walls, as seen in Figures \ref{fig: eg, B2}, \ref{Fig: A3 example}, and \ref{Figure03}.} We assign a pseudo-torsion class to each pseudo-chamber. We also assign two pseudo-torsion classes to each wall, one for each side of the wall, the ``positive'' side and the ``negative'' side. When the wall is ``thin'' we show that the pseudo-torsion class on the positive side of the wall is a minimal extension of the one on the negative side of the wall (Corollary \ref{cor: for thin walls, ov P is min ext of P}). The minimal extending module is the pseudo-brick which labels the wall. (Compare with the classical result in \cite{BCZ}.) The main theorem here is that the pseudo-torsion class assigned to each point in the stability space $\RR^n$ is constant on each pseudo-chamber and on each green component of the union of pseudo-walls (Definition \ref{def: side components of walls}):

{\color{black}
\begin{customthm}{C}[Theorem \ref{thm: P(theta) is constant on path components} and Theorem \ref{thm: P(theta) are different in different path components}]
    For $\theta_1$ and $\theta_2$ in a chamber, the associated pseudo-torsion classes are equal: $\cP(\theta_1) = \cP(\theta_2)$. Moreover, if $\theta_1$ and $\theta_2$ lie in distinct path components, then $\cP(\theta_1) \neq \cP(\theta_2)$.
\end{customthm} }

We also consider ``relative'' pseudo-torsion classes in the pseudo-wide subcategory which labels each thick wall. Fortunately, ``relative'' pseudo-wide subcategories are again pseudo-wide subcategories which helps with induction on deeper ``layers'' of refinements of thick walls.

We give several examples in this section. In one of these, ``ghosts'' and their corresponding ``derived objects'' will appear. Ghosts are the remnants of objects of $mod\text-\Lambda$ which do not lie in the torsion class $\cG$. If $A$ is one of these missing objects, we embed $A$ into an object $B\in\cG$ with quotient $C=B/A$. The ``ghost'' is the map $B\to C$ considered as an object of the ``pseudo-derived category'' of $\cG$. The same missing object can have several ghosts. Ghosts originated in \cite{GrInvRedrawn} and were further explained in \cite{MoreGhosts}.

In Section \ref{new sec3}, given a ``green path'', we construct a ``strict'' Harder-Narasimhan stratification of $\cG$. Each object of $\cG$ will have a unique strict filtration with subquotients in the ``slices'' of the HN-stratification:

{\color{black}
\begin{customthm}{D}[Theorem \ref{thm: HN stratification} and Proposition \ref{prop: uniqueness of HN}]
    A ``smooth green path" gives rise to a strict HN-stratification of $\cG$ in which each module $M\in\cG$ has a unique strict HN-filtration.
\end{customthm} }

The objects in the slices can be arranged to form a ``strict forward Hom-orthogonal sequence'' (FHO) (Definition \ref{def: FHO}). In the presentation of this idea, it will be easier to construct the FHO-sequence first and the HN-stratification as a consequence. This discussion is based on \cite{Bridgeland} and \cite{Linearity}.

{\color{black} We use our results on green paths to study functorially finite pseudo-torsion classes, which are pseudo-torsion classes $\cP$ such that both $\cP$ and $\cP^{\perp}$ (the strict right Hom-orthogonal category in $\cG$) are finitely generated. We prove the following.}

{\color{black}
\begin{customthm}{E}
    For $\theta$ in the stability space, we have the following:
    \begin{itemize}
        \item If $\theta$ lies in a pseudo-chamber $\cU$ and $\cP(\theta)$ is finitely generated, then $\cU$ has a nonempty interior and $\theta$ either lies in the interior or is a ``green boundary point'' of $\cU$. (Corollary \ref{cor: green boundary})

        \item If $\theta$ lies in a pseudo-chamber $\cU$ and $\cQ(\theta)$ is finitely generated, then $\cU$ has a nonempty interior and $\theta$ either lies in the interior or is a ``red boundary point'' of $\cU$. (Corollary \ref{cor: red boundary})

        \item If $\theta$ lies in a pseudo-chamber $\cU$ and both  $\cP(\theta)$ and $\cQ(\theta)$ are finitely generated, then $\cU$ is an open set. (Theorem  \ref{thm: ff P,Q implies U open})
    \end{itemize}
\end{customthm} }

Section \ref{new sec4} gives a preview of our next paper in which we explain how ``thick'' walls create an issue that we solve by ``resolving'' each thick wall into {\color{black}``quasi-thin''} layers and obtain an infinite dimensional stability diagram. We show that this (in general) infinite dimensional vector space has a wall and chamber structure and that a relative pseudo-torsion class is assigned to each chamber. We call this the ``second layer'' of a stratified stability space. This will have thick walls which we resolve to obtain the ``third layer'', etc. This infinitely iterated refinement process is necessary to construct a maximal green sequence for our original category $\cG$. ``Ghosts'', as they are currently defined, appear only in the first layer for the simple reason that, in subsequent layers, we do not leave any of the objects out. However, we may get ``relative ghosts'' in the subsequent layers. 

\ki{Investigation of stability diagrams for torsion classes originated in a paper about the group $\pi_0\cC(M)$ of the concordance space of $M$. This is the space of diffeomorphisms of $M\times I$ where $M$ is a compact smooth manifold and $I=[0,1]$. In \cite{HW}, \cite{H2} it is shown that for $\dim M\ge6$, there is an exact sequence:
\[
    Wh_1^+(\pi_1M;\ZZ_2\oplus \pi_2M)\xrightarrow j \pi_0\cC(M)\to Wh_2(\pi_1M)\to 0
\]
where $Wh_2(\pi)$ is a certain quotient of the algebraic K-theory group $K_2\ZZ \pi$ for any group $\pi$ and $Wh_1^+(\pi_1M;X)$ is a quotient of $H_0(\ZZ[\pi];X[\pi])$ for any group $\pi$ and any $\pi$-module $X$. In \cite{GrInvRedrawn} a formula for the kernel of $j$ was given in the special case when $\pi_2M=0$. In that case, the kernel of $j$ is the image of
\[
    \chi_c:K_3\ZZ\pi\to H_0(\ZZ\pi;\ZZ_2\pi)
\]
composed with the projection $H_0(\ZZ\pi;\ZZ_2\pi)\to Wh_1^+(\pi;\ZZ_2)$. Stability diagrams for torsion classes are involved in the formula for $\chi_c$. See Section \ref{sec 5: Alg K-theory} for details.}


\section{Basic definitions}\label{sec1}

Let $\Lambda$ be a finite dimensional algebra over any field. Let $n$ be the rank of $\Lambda$. Then, up to isomorphism, there are $n$ simple modules $S_1,S_2,\cdots,S_n$ covered by $n$ projective modules $P_1,\cdots,P_n$. We take only finitely generated right $\Lambda$-modules. The category of such modules is usually denoted $mod\text-\Lambda$.

We consider a fixed \textbf{torsion class} $\cG$ in $mod\text-\Lambda$. This is a class of finitely generated $\Lambda$-modules which is closed under extension and under taking quotients. This means that for any short exact sequence of $\Lambda$-modules
\[
    0\to A\to B\to C\to 0
\]
if $A,C$ are in $\cG$ then so is $B$. And if $B\in \cG$ then $C\in\cG$. 

The kernel $A$ of a morphism $B\to C$ in $\cG$ might not be in $\cG$ and this is a problem that we need to get around. We will construct a new category: the ``strict category $\cG$'' with the same objects as $\cG$ but with only ``strict morphisms''. This new category will have good properties. Every morphism will have kernel, image and cokernel in $\cG$. Using these new morphisms we will construct ``pseudo-torsion classes'' and ``pseudo-torsionfree classes''.

\subsection{Strict morphisms and strict extensions}\label{ss11: strict morphisms}

\begin{defn}\label{def: strict subobject}
    Given an object $B$ in a torsion class $\cG$, a \textbf{subobject} is a submodule $B'$ of $B$ which lies in $\cG$. This includes $0$ and $B$. By a \textbf{strict subobject} we mean a subobject $A\subseteq B$ so that $A\cap B'$ also lies in $\cG$ for any other subobject $B'$ of $B$. We use the notation $A\cof B$ to indicate a monomorphism whose image is a strict subobject of $B$. We call such a monomorphism a \textbf{cofibration}. 
\end{defn}

Note that any simple subobject of an object $B\in\cG$ will be a strict subobject for trivial reasons.

\begin{lem}\label{lem: parallel inclusions}
    If $A$ is a strict subobject of $B$ and $B'$ is another subobject of $B$. Then $A\cap B'$ is a strict subobject of $B'$.
\end{lem}

\begin{proof}
    Any subobject $X$ of $B'$ is also a subobject of $B$. So, $A\cap X$ lies in $\cG$. But $(A\cap B')\cap X=A\cap X$. So, $A\cap B'\cof B'$.
\end{proof}

\begin{thm}\label{thm: composition of cofibrations}
    Any composition of cofibrations is a cofibration.
\end{thm}

\begin{proof}
    Suppose that $A\cof B\cof C$ and $C'$ is a subobject of $C$. Then $B\cap C'$ is a subobject of $B$. So, $A\cap (B\cap C')=A\cap C'$ is a subobject of $A$. So, $A\cof C$.
\end{proof}

Note that, by Lemma \ref{lem: parallel inclusions}, we get an induced sequence of cofibrations: 
\[
A\cap C'\cof B\cap C'\cof C'.
\]

\begin{defn}\label{def: strict quotients}
    Given $B$ in a torsion class $\cG$, a \textbf{strict quotient} of $B$ is defined to be $B/A$ where $A$ is a strict subobject of $B$. We denote the quotient map by $B\onto B/A$ and call it a \textbf{fibration}. The resulting short exact sequence
    \[
    A\cof B\onto B/A
    \]
    will be called a \textbf{strict exact sequence} and
    we say that $B$ is a \textbf{strict extension} of $A$ by $B/A$. Thus, a short exact sequence is strict if and only if both maps are strict. 
\end{defn}

We will show that fibrations are closed under composition. The proof will be very similar to the proof of the Main Lemma which it will be convenient to state and prove now.

\begin{lem}[Main Lemma]\label{lemma A}
    Let $0\to A\to B\xrightarrow p C\to 0$ be a short exact sequence of objects in $\cG$ and let $0\to A'\to B'\to C'\to 0$ be an exact sequence of subobjects of $A,B,C$. I.e., $A'=A\cap B'$ and $C'=pB'$. If $B'$ is a strict subobject of $B$ then $A'\cof A$ and $C'\cof C$ are strict. In other words, $B'/A'\cof B/A$ is strict.

\end{lem}

\begin{proof}
    We have already shown that $A'\cof A$. So, suppose $X$ is any subobject of $C$. Let $p^\ast X$ be the pull-back of $X$ to $B$. This gives a short exact sequence
    \[
        0\to A\to p^\ast X\to X\to 0
    \]
    Since $A,X\in\cG$, so is $p^\ast X$. Then $p^\ast X\cap B'\in \cG$ since $B'\cof B$. We claim that $X\cap C'$ is a quotient of $p^\ast X\cap B'$ and therefore also lies in $\cG$. This statement follows from the Snake Lemma applied to the following diagram with exact rows.
       \[
   \xymatrixrowsep{10pt}\xymatrixcolsep{10pt}
\xymatrix{
0\ar[r]& A' \oplus A\ar[d]\ar[r] &
	B'\oplus p^\ast X\ar[d]\ar[r] &
	C'\oplus X\ar[d]\ar[r]
	& 0\\
0\ar[r] & A \ar[r]& 
	B \ar[r]&
	C\ar[r]&0 
	} 
   \]
   The sequence of kernels is the exact sequence
   \[
   	0\to A'\to B'\cap p^\ast X\to C'\cap X\to 0.
   \]
   So, $C'\cap X\in\cG$ making $C'\cof C$.
\end{proof}

This lemma is illustrated by the following diagram which we call the \textbf{main diagram}.
\begin{equation}\label{eq: main diagram}
\xymatrixrowsep{10pt}\xymatrixcolsep{10pt}
\xymatrix{
& 0\ar[d] & 0\ar[d] & 0\ar[d]\\
0\ar[r]& A'\ar[d]\ar[r] &
	B'\ar[d]_i\ar[r] &
	C'\ar[d]\ar[r]
	& 0\\
0\ar[r]& A\ar[d]\ar[r] &
	B\ar[d]\ar[r]^p &
	C\ar[d]\ar[r]
	& 0\\
0\ar[r]& A''\ar[d]\ar[r] &
	B''\ar[d]\ar[r] &
	C''\ar[d]\ar[r]
	& 0\\
& 0 & 0 & 0	} 
\end{equation}
We observe that the entire diagram is determined by $B$, the subobject $B'$ and the epimorphism $p:B\to C$.
Indeed, $A=\ker p$, $C'$ is the image of $B'$ in $C$, $A'=\ker p\circ i$. $A'',B'',C''$ are the quotients of the vertical monomorphisms. If either $B'$ is a strict subobject of $B$ or $C$ is a strict quotient of $B$ then we will also have that $A'\in\cG$.

Lemma \ref{lemma A} states that if the middle vertical sequence $B'\cof B\onto B''$ is strict then so are the ``parallel'' sequences $A'\cof A\onto A''$ and $C'\cof C\onto C''$. By symmetry we can also say that, if the middle horizontal sequence $A\cof B\onto C$ is strict then so are the top and bottom sequences.

\begin{lem}\label{lem: all maps strict}
    In the main diagram \eqref{eq: main diagram}, if $B'$ is a strict subobject of $B$ and $C$ is a strict quotient of $B$ then all morphisms in \eqref{eq: main diagram} are strict.
\end{lem}

Since the main diagram is determined by the middle horizontal and middle vertical sequences, we have the following consequences of the Main Lemma \ref{lemma A}.

\begin{prop}\label{prop: duality of strict ses}
Let $0\to A\to B\to C\to 0$ be a short exact sequence of objects in $\cG$. Then any strict quotient $B''$ of $B$ is an extension of a strict quotient $A''$ of $A$ by a strict quotient $C''$ of $C$.
\end{prop}

We define an \textbf{admissible quotient} $B''$ of an object $B\in\cG$ as $B''=B/B'$ where $B'$ is a subobject of $B$ ($B'\in\cG$). (Recall that $B''$ is a strict quotient of $B$ if and only if $B'$ is a strict subobject of $B$.)

\begin{prop}\label{prop: old def of admis seq}
    Let $A\cof B\onto C$ be a strict extension. Then

    a) For any subobject $B'\subseteq B$ there are subobjects $A'\subseteq A$ and $C'\subseteq C$ making a strict extension $A'\cof B'\onto C'$.

    b) Any admissible quotient $B''$ of $B$ is an extension of an admissible quotient $A''$ of $A$ by an admissible quotient $C''$ of $C$. Furthermore, the extension $A''\cof B''\onto C''$ is a strict extension.
\end{prop}

{\color{black}
Combining Propositions \ref{prop: duality of strict ses} and \ref{prop: old def of admis seq} we obtain the following.

\begin{cor}\label{cor: strict quotient of strict sequence}
    Let $A\cof B\onto C$ be a strict extension. Then any strict quotient $B''$ of $B$ is a strict extension of a strict quotient $A''$ of $A$ by a strict quotient $C''$ of $C$.
\end{cor}
}

\begin{thm}\label{thm: composition of fibrations}
    Any composition of fibrations is a fibration.
\end{thm}

\begin{proof} 
    Suppose that $f:B\onto C$ and $g:C\onto D$ are fibrations. Let $A=\ker f$, $C'=\ker g$. Then we have a short exact sequence of kernels:
    \[
        0\to \ker f\to \ker gf\to \ker g\to0.
    \]
    This is a subsequence of the exact sequence $0\to A\to B\to C\to 0$ and we need to prove that $B'=\ker gf$ is a strict subobject of $B$. So, let $X$ be any other subobject of $B$. Then we get an induced short exact sequence $0\to A\cap X\to X\to f(X)\to 0$. As in the proof of the Main Lemma \ref{lemma A}, we take direct sum with the kernel sequence, giving the following diagram with exact rows.
           \[
   \xymatrixrowsep{10pt}\xymatrixcolsep{10pt}
\xymatrix{
0\ar[r]& A \oplus A\cap X\ar[d]\ar[r] &
	B'\oplus X\ar[d]\ar[r] &
	C'\oplus f(X)\ar[d]\ar[r]
	& 0\\
0\ar[r] & A \ar[r]& 
	B \ar[r]&
	C\ar[r]&0 
	} 
   \]
    By the Snake Lemma, we get an exact sequence of kernels of the vertical maps:
    \[
    0\to A\cap X\to B'\cap X\to C'\cap f(X)\to 0.
    \]
    But, $A\cap X\in \cG$ since $A\cof B$ and $C'\cap f(X)\in \cG$ since $C'\cof C$. Therefore, the extension $B'\cap X\in\cG$. So, $B'=\ker gf\cof B$ making $gf:B\onto D$ a fibration.
\end{proof}

\begin{defn}\label{def: strict morphism}
    We define a \textbf{strict morphism} to be a morphism $f:A\to B$ between two objects of $\cG$ with $\ker f$ in $\cG$ so that $\ker f$ is a strict subobject of $A$ and $\im f$ is a strict subobject of $B$. Equivalently, $f$ is the composition of a fibration $A\onto X$ and cofibration $X\cof B$.  
\end{defn}

Note that, by definition, strict morphisms have kernel, image and cokernel in $\cG$.

\begin{thm}\label{thm: strict torsion class is a category}
    If $f:A\to B$ and $g:B\to D$ are strict maps, so is $g\circ f:A\to D$.
\end{thm}

\begin{proof} Let $B'$ be the image of $f$ and let $C$ be the image of $g$. Then we have the lower part of the following commuting diagram where ascending maps are fibrations and descending maps are cofibrations.
        \[
    \xymatrixrowsep{10pt}\xymatrixcolsep{20pt}
\xymatrix{
&& X\ar[dr]^{\exists \widetilde{f_2}}\\
& B' \ar[dr]^{f_2}\ar[ur]^{\exists \widetilde{g_1}} && C\ar[dr]^{g_2}\\
A \ar[rr]^f\ar[ur]^{f_1}& &
	B \ar[rr]^g\ar[ur]^{g_1}& & D
	}
    \]
    We claim that there is an object $X$ as indicated in the diagram. Then composing $A\onto B'\onto X$ gives a fibration $A\onto X$ by Theorem \ref{thm: composition of fibrations} and composition of $X\cof C\cof D$ gives a cofibration $X\cof D$ by Theorem \ref{thm: composition of cofibrations}. Then $gf:A\to D$ would factor as $A\onto X\cof D$ making it a strict morphism. 
    
    The main diagram \eqref{eq: main diagram} shows that $X$ exists. The main diagram can be generated from any cofibration $B'\cof B$ and any fibration $B\onto C$. By Lemma \ref{lem: all maps strict} all arrows in the main diagram will be strict. So, we can let $X=C'$ in the main diagram and all our claims are satisfied. Our theorem is proved.
\end{proof}

We also need ``strict subquotients'' of objects $M\in\cG$. This has several equivalent definitions.

{
\begin{thm}\label{def/thm: subquotient}
    The following characterizations of a \textbf{strict subquotient} of $M\in\cG$ are equivalent.

    (a) A strict subquotient of $M$ is a strict quotient of a strict subobject of $M$

    (b) A strict subquotient of $M$ is a strict subobject of a strict quotient of $M$.
    
    (c) Strict subquotients of $M$ are $B/A$ where $A,B$ are strict subobjects of $M$ so that $A\subseteq B$.
\end{thm}

\begin{proof}
We will show that the collection of subquotients $B/A$ of $M$ described in (a), (b), (c) are the same.

In (a) we have a strict subobject $B\cof M$ and a strict quotient $B/A$ of $B$. This implies $A\cof B$ is strict. So, $A\cof M$ is strict by Theorem \ref{thm: composition of cofibrations}. By the Main Lemma \ref{lemma A}, $B/A$ is a strict subobject of $M/A$. So, we have (b).

In (b) we have $A\cof M$ strict with strict quotient $M/A$ and $B/A\cof M/A$. Thus,
\[
    M/A\onto \frac{M/A}{B/A}=M/B
\]
is a fibration. So, the composition $M\onto M/A\onto M/B$ is strict. So, $B\cof M$ is a cofibration and $B/A$ satisfies (c).

In (c) we have $A,B$ strict subobjects of $M$ with $A\subset B$. By Lemma \ref{lem: parallel inclusions}, we know that $A\cof B$ is strict. So, $B/A$ is a strict quotient of $B$ which is a strict subobject of $M$. So, (a) is satisfied. 

Thus $(a)\subseteq (b)\subseteq (c)\subseteq (a)$. So the three sets of subquotients are equal.
\end{proof}

\begin{rem}\label{rem: inherited subquotients}
Given $A\cof B\cof M$ as above suppose that $X,Y$ are subobjects of $M$ so that $X\subseteq A\subseteq B\subseteq Y$. Then, by Lemma \ref{lem: parallel inclusions}, $A,B$ are strict subobjects of $Y$ and, by the Main Lemma \ref{lemma A}, $A/X,B/X$ are strict subobjects of $Y/X$. Therefore, $(B/X)/(A/X)\cong B/A$ is a strict subquotient of $Y/X$.  
\end{rem}
}

{%
\begin{eg}\label{eg: A3 example}
Let $\Lambda$ be the path algebra of the quiver $1\ot 2\ot 3$ and let $\cG$ be the torsion class of all modules $M$ so that $\Hom(M,S_1)=0$. Here is the Auslander-Reiten quiver of $\Lambda$ with objects in $\cG$ circled. 
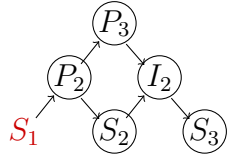
\begin{center}
\begin{tikzpicture}[scale=1.2]
{ 
\begin{scope}[xshift=4cm,yshift=-.5cm]
\draw[red!80!black] (0,0) node{$S_1$}; 
\draw(0.5,0.6)circle[radius=2.4mm];
\draw (0.5,0.6) node{$P_2$}; 
\draw (1.5,0.6) node{$I_2$};
\draw(1.5,0.6)circle[radius=2.4mm];
\draw (1,1.2) node{$P_3$};
\draw(1,1.2)circle[radius=2.4mm];
\draw (1,0) node{$S_2$}; 
\draw(1,0)circle[radius=2.4mm];
\draw (2,0) node{$S_3$}; 
\draw(2,0)circle[radius=2.4mm];
\draw[->] (0.13,0.15)--(.33,.4);
\draw[->] (1.13,0.15)--(1.33,0.4);
\draw[->] (0.63,0.75)--(.83,1);
\draw[->] (0.63,0.37)--(.83,.14);
\draw[->] (1.63,.37)--(1.83,.14);
\draw[->] (1.13,.97)--(1.33,.74);
\end{scope}
}
\end{tikzpicture}
\end{center}

The following example of the main diagram shows that the almost split sequence $0\to P_2\to P_3\oplus S_2\to I_2\to 0$ and the split exact sequence $0\to P_3\to P_3\oplus S_2\to S_2\to0$ are not strict (since $S_1\notin\cG$).
\[
\xymatrixrowsep{12pt}\xymatrixcolsep{12pt}
\xymatrix{
& 0\ar[d] & 0\ar[d] & 0\ar[d]\\
0\ar[r]& S_1\ar[d]\ar[r] &
	P_3\ar[d]\ar[r] &
	I_2\ar[d]^=\ar[r]
	& 0\\
0\ar[r]& P_2\ar[d]\ar[r] &
	P_3\oplus S_2\ar[d]\ar[r] &
	I_2\ar[d]\ar[r]
	& 0\\
0\ar[r]& S_2\ar[d]\ar[r]^= &
	S_2\ar[d]\ar[r] &
	0\ar[d]\ar[r]
	& 0\\
& 0 & 0 & 0	} 
\]
On the other hand, any extension $A\cof B\onto C$ in $\cG$ will be strict if $S_1$ is not a submodule of $A$. The converse of this statement does not hold: $P_2\cof P_3\onto S_3$ is strict even though $S_1\subset P_2$ since $P_2$ is the only proper subobject of $P_3$.

Also $S_2$ is a strict subobject of $P_3\oplus S_2$ since $S_2$ is simple.
\end{eg}
}

\begin{defn}\label{def: strict G}
    We consider the category having the same objects as $\cG$ but with only strict morphisms. We call this the \textbf{strict category $\cG$}.
\end{defn}

There is one subtle point we need to settle to avoid confusion for the words ``kernel'' and ``cokernel''. By default, $\ker f$ and $\coker f$ are defined to be the kernel and cokernel of $f$ in $\cG$.

\begin{prop}\label{prop: ker and coker in G}
    Let $f:A\to B$ be a strict morphism in $\cG$. Then $f$ has kernel and cokernel in the strict category $\cG$ and these agree with the kernel and cokernel of $f$ in $\cG$. In other words:
    \begin{enumerate}
        \item Given a strict morphism $g:X\to A$ so that $fg=0$, the induced map $\widetilde g:X\to \ker f$ is strict.
       \item For any strict morphism $h:B\to Y$ so that $hf=0$, the induced map $\overline h:\coker f\to Y$ is strict.
    \end{enumerate}
\end{prop}

\begin{proof}
    (1) follows from Lemma \ref{lem: parallel inclusions}: Since $g$ is strict, its image is a strict subobject of $A$. Since $fg=0$, $\im g\subset \ker f$. By Lemma \ref{lem: parallel inclusions}, $\im g$ is a strict subobject of $\ker f$. So, the lifting $\widetilde g:X\to \ker f$ is strict. 

    (2) Similarly, if $h:B\to Y$ factors through $\coker f$, we have a factorization of $h$ as the composition $B\xrightarrow p \coker f\xrightarrow{\overline h} Y$. This gives an exact sequence of kernels
    \[
    0\to \ker p\to \ker h\to \ker\overline h\to 0
    \]
    which is a subsequence of the exact sequence
    \[
    0\to \ker p\to B\to \coker f\to 0.
    \]
    Since $h$ is strict, $\ker h$ is a strict subobject of $B$. By the Main Lemma \ref{lemma A}, this implies that $\ker\overline h$ is a strict subobject of $\coker f$. Therefore, the induced map $\overline h:\coker f\to Y$ is strict.

    Thus, the notions of kernel and cokernel agree in $\cG$ and the strict category $\cG$. 
\end{proof}

\subsection{Pseudo-torsion pairs}\label{ss12: p-tor pair}

Now we define ``pseudo-torsion classes'' and ``pseudo-torsionfree classes'' and show that they form strict Hom-perpendicular pairs.

\begin{defn}\label{def: pseudo-torsion}
A nonempty set of objects $\cP\subseteq\cG$ is called a \textbf{pseudo-torsion class} if the following hold.
\begin{enumerate}
\item $\cP$ is closed under strict quotients, i.e., if $A\onto B$ is a fibration and $A\in\cP$ then so is $B$. In particular, $0\in\cP$.
\item $\cP$ is closed under strict extensions.
\end{enumerate}
\end{defn}

\begin{defn}\label{def: pseudo-torsionfree}
A nonempty class $\cQ\subseteq\cG$ is called a \textbf{pseudo-torsionfree class} if the following hold.
\begin{enumerate}
\item $\cQ$ is closed under strict subobjects, i.e., if $A\cof B$ is a cofibration and $B\in\cQ$ then so is $A$. In particular, $0\in\cQ$.
\item $\cQ$ is closed under strict extensions.
\end{enumerate}
\end{defn}

\begin{defn}\label{def: perp class}
    For any $\cX\subseteq\cG$, its \textbf{right perpendicular class} $\cX^\perp$ is defined to be the class of all $Y\in \cG$ so that there are no nonzero strict morphisms $X\to Y$ for any $X\in\cX$. Similarly, for any $\cY\subseteq\cG$, the \textbf{left perpendicular class} $\,^\perp \cY$ is the class of all $X\in\cG$ so that there are no nonzero strict morphisms from $X$ to any object in $\cY$.
\end{defn}

\begin{lem} \label{lem: perps are pseudo torsion (free) classes}
    Let $\cX \subset \cG$ be a subclass of objects in a torsion class $\cG$. 
    \begin{enumerate}
        \item $\,^\perp \cX$ is a pseudo-torsion class.
        \item $\cX^{\perp}$ is a pseudo-torsionfree class.
    \end{enumerate}
\end{lem}

\begin{proof}
    We prove (1) as the proof of (2) is analogous. Let $M \in \,^\perp \cX$ and $f:M\onto N$ be a strict quotient of $M$. Suppose that $g: N \rightarrow X$ is a nonzero strict morphism where $X \in \cX$. Then, the composition $M\onto N\to X$ is a nonzero strict morphism contradicting the assumption that $M \in \,^\perp \cX$. Therefore, $\,^\perp \cX$ is closed under strict quotients.
    
   Now suppose $M,N \in \,^\perp \cX$ and $N \cof E \onto M$ is a strict extension. Let $g:E \rightarrow X$ be a strict morphism. The composition $N\cof E\to X$ is strict and therefore must be 0 since $N \in \,^\perp \cX$. So, $g:E\to X$ factors through a morphism $\overline g: M\to X$. By Proposition \ref{prop: ker and coker in G}, $\overline g$ is strict and therefore equal to 0. So, $g=0$ and $E\in \,^\perp \cX$.
   
    Since $\,^\perp \cX$ is closed under strict quotients and strict extensions, it is a pseudo-torsion class.
\end{proof}

\begin{thm}\label{thm: duality between P,Q}
    The right perpendicular class $\cP^\perp$ of any pseudo-torsion class $\cP$ is a pseudo-torsionfree class. The left perpendicular class of any pseudo-torsionfree class is a pseudo-torsion class. Furthermore, this gives a bijection between the set of pseudo-torsion classes and the set of pseudo-torsionfree classes, i.e.: \[
    \,^\perp(\cP^\perp)=\cP.\]
\end{thm}

\begin{proof}
    Let $\cP\subset \cG \subset mod$-$\Lambda$ be a pseudo-torsion class in a torsion class in $mod$-$\Lambda$. By definition, $\cP \subset \,^\perp(\cP^\perp)$. 
    
    To show the other inclusion, let $X \in \,^\perp(\cP^\perp)$. Suppose by contradiction that $X\notin \cP$. Take $X$ to be minimal. Then $X \notin \cP^{\perp}$, so there exists $Y \in \cP$ and a nonzero strict morphism $Y \overset{f}{\rightarrow} X$. Consider the quotient $X/\im f$. Then we have a strict extension $\im f \cof X \twoheadrightarrow X/\im f$. Since $Y\in \cP$, its image $\im f\in\cP$. Since $\,^\perp(\cP^\perp)$ is a torsion class by Lemma \ref{lem: perps are pseudo torsion (free) classes}, we have that $X/\im f \in \,^\perp(\cP^\perp)$. Since $f$ is nonzero, $X/\im f$ is smaller than $X$. By induction on the size of $X$ we have that $X/\im f\in\cP$. Therefore, the strict extension $X$ of $\im f$ by $X/\im f$ lies in $\cP$ contradicting the assumption that $X\notin \cP$. Therefore, $\cP=\,^\perp(\cP^\perp)$ as claimed. 
\end{proof}

    We call $(\cP,\cP^\perp)$ a \textbf{pseudo-torsion pair}.

\begin{thm}\label{thm: tM and fM}
Given a pseudo-torsion pair $(\cP,\cQ)$ and $M\in \cG$, there exists a unique strict extension
\[
	tM\cof M\onto fM
\] 
where $tM\in \cP$ and $fM\in \cQ$. Furthermore, any strict morphism $M\to N$ gives unique strict morphisms $tM\to tN$ and $fM\to fN$ making the following diagram commute.
\[
\xymatrixrowsep{10pt}\xymatrixcolsep{10pt}
\xymatrix{
0\ar[r]& tM\ar[d]\ar[r] &
	M\ar[d]\ar[r] &
	fM\ar[d]\ar[r]
	& 0\\
0\ar[r] & tN \ar[r]& 
	N \ar[r]&
	fN\ar[r]&0 
	}
\]
\end{thm}

\begin{proof}
Let $tM\cof M$ be a maximal strict subobject of $M$ which lies in $\cP$. Then we claim that $M/tM\in \cP^\perp=\cQ$. To see this, suppose not. Then there is a nonzero strict morphism $P\to M/tM$. Its image $X$ is a strict subobject of $M/tM$ which lies in $\cP$ since $\cP$ is closed under strict quotients. But then, the inverse image $Y$ of $X$ in $M$ is a strict subobject of $M$ which has a strict filtration $tM\cof Y\onto X$. So, $Y\in\cP$ contradicting the maximality of $tM$.

To show uniqueness, suppose that $t'M\cof M\onto f'M$ is another strict filtration with $t'M\in \cP$ and $f'M\in \cQ$. Then $t'M\to M\to fM$ is a strict morphism which must be zero. So, $t'M\subseteq tM$. Similarly, $tM\subseteq t'M$. So, $t'M=tM$, proving uniqueness.

Similarly, given any strict morphism $f:M\to N$, the composition $tM\to M\to N\to fN$ must be zero. So, $f(tM)\subseteq tN$ and the induced map $tM\to tN$ is strict by Proposition \ref{prop: ker and coker in G}. Using Proposition \ref{prop: ker and coker in G} again, this induces a strict map $fM\to fN$ as claimed. 
\end{proof}

We also can specify the smallest pseudo-torsion class containing a set of objects. Let $\cX \subset \cG$ be a class of objects in the torsion class $\cG$. We denote by $Filt(\cX)$ the class of objects $M \in \cG$ such that $M$ admits a \textbf{strict $\cX$ filtration}, by which we mean a strict filtration $0 = M_0 \cof M_1 \cof \cdots \cof M_k=M$ such that each $M_i$ is a strict subobject of $M$ and each subquotient $M_i/M_{i-1}$ is a strict quotient of an object from $\cX$.

\begin{prop}\label{prop: filt X}
    The smallest pseudo-torsion class containing $\cX$ is $Filt(\cX)$.
\end{prop}

\begin{proof}
    We begin by showing $Filt(\cX)$ is a pseudo-torsion class. Let $M\in Filt(\cX)$ and $0 = M_0 \cof M_1 \cof \cdots \cof M_n=M$ be a strict $\cX$ filtration of $M$. Suppose $f:M \twoheadrightarrow N$ is a fibration onto $N\in\cG$. Then we attain a filtration $0 = N_0 \subset N_1 \subset \cdots \subset N_n = N$ where $N_i = {M_i + \ker f \over \ker f}$. We will show this is a strict $\cX$ filtration of $N$ by backward induction. Since $\cG$ is closed under sums and quotients, we have that $N_i \in \cG$ for all $i$. Consider the short exact sequence $M_{n-1} \cof M \twoheadrightarrow M/M_{n-1}$. Since all modules in this sequence are in $\cG$ and $M\twoheadrightarrow N$ is a fibration, by Proposition \ref{prop: duality of strict ses}, we attain the following commuting diagram in which all vertical morphisms are fibrations:

    \[
\xymatrixrowsep{12pt}\xymatrixcolsep{12pt}
\xymatrix{
        0 \ar[r] & M_{n-1} \ar[r] \ar[d] & M \ar[r] \ar[d] & M/M_{n-1} \ar[r] \ar[d] & 0 \\ 0 \ar[r] & A \ar[r] & M/\ker f \ar[r] & C \ar[r] & 0
    }
    \]
\noindent
    where $A$ and $C$ are strict quotients of $M_{n-1}$ and $M/M_{n-1}$ respectively. It follows that $A$ and $C$ are uniquely determined and we can extend this to the main diagram (\ref{eq: main diagram}):

    \[
\xymatrixrowsep{12pt}\xymatrixcolsep{12pt}
\xymatrix{
& 0\ar[d] & 0\ar[d] & 0\ar[d]\\
0\ar[r]& \ker f \cap M_{n-1} \ar[d]\ar[r] &
	\ker f\ar[d]\ar[r] &
	{M_{n-1} + \ker f \over M_{n-1}} \ar[d]\ar[r]
	& 0\\
0\ar[r]& M_{n-1} \ar[d]\ar[r] &
	M \ar[d]\ar[r] &
	M/M_{n-1} \ar[d]\ar[r]
	& 0\\
0\ar[r]& {M_{n-1} + \ker f \over \ker f} \ar[d]\ar[r] &
	M/\ker f \ar[d]\ar[r] &
	{M \over M_{n-1} + \ker f}\ar[d]\ar[r]
	& 0\\
& 0 & 0 & 0	} 
\]

Since $M/M_{n-1}$ is a strict quotient of some $X \in Filt(\cX)$, $M/M_{n-1} \twoheadrightarrow {M \over M_{n-1} + \ker f}$ is a fibration, and composition of fibrations is strict, we have that ${M \over M_{n-1} + \ker f} \cong N/N_{n-1}$ is a strict quotient of $X$. Moreover, we have that $M_{n-1} \twoheadrightarrow {M_{n-1} + \ker f \over \ker f} \cong N_{n-1}$ is a fibration.

Now suppose that $N_{i+1}/N_{i}$ is a strict quotient of an object $X \in Filt(\cX)$ and that $M_{i} \twoheadrightarrow N_{i}$ is a fibration. Consider the short exact sequence $M_{i-1} \cof M_i \twoheadrightarrow M_i/M_{i-1}$. Using an analogous argument, we attain the following diagram:

 \[
\xymatrixrowsep{12pt}\xymatrixcolsep{12pt}
\xymatrix{
& 0\ar[d] & 0\ar[d] & 0\ar[d]\\
0\ar[r]& \ker f \cap M_{i-1} \ar[d]\ar[r] &
	\ker f\cap M_i \ar[d]\ar[r] &
	{\ker f \cap M_i \over M_{i-1}} \ar[d]\ar[r]
	& 0\\
0\ar[r]& M_{i-1} \ar[d]_{\pi_0} \ar[r] &
	M_i \ar[d]_{\pi_1} \ar[r] &
	M_i/M_{i-1} \ar[d] _{\pi_2} \ar[r]
	& 0\\
0\ar[r]& {M_{i-1} + \ker f \over \ker f} \ar[d]\ar[r] &
	{M_i + \ker f \over \ker f} \ar[d]\ar[r] &
	{M_i + \ker f \over M_{i-1} + \ker f}\ar[d]\ar[r]
	& 0\\
& 0 & 0 & 0	} 
\]
In particular, we have that $\pi_0, \pi_1,$ and $\pi_2$ are fibrations. It follows that ${M_i + \ker f \over M_{i-1} + \ker f} \cong N_i/N_{i-1}$ is a strict quotient of an object in $\cX$ and that $M_{i-1} \twoheadrightarrow {M_{i-1} + \ker f \over \ker f} \cong N_{i-1}/N_i$ is a fibration. By backward induction, we conclude $0 = N_0 \cof N_1 \cof \cdots \cof N_n = N$ is a strict $\cX$ filtration of $N$.

Now let $M,N \in Filt(\cX)$ and $M \cof E \overset{\pi}{\twoheadrightarrow} N$ be a strict short exact sequence. Then we attain a strict $\cX$ filtration $0 = M_0 \cof M_1 \cof \cdots \cof M = \pi^{-1}(N_0) \cof \pi^{-1}(N_1) \cof \cdots \cof \pi^{-1}(N) = E$ of $E$ where $0=M_0 \cof M_1 \cof \cdots \cof M$ and $0=N_0 \cof N_1 \cof \cdots \cof N$ are strict $\cX$ filtrations of $M$ and $N$ respectively. Thus $Filt(\cX)$ is closed under strict extensions and therefore is a pseudo-torsion class.

Note that each module in $Filt(\cX)$ is an iterated strict extension of strict quotients of objects from $\cX$. Since every pseudo-torsion class containing $\cX$ must contain all objects that are iterated strict extensions of strict quotients of objects of $\cX$, it follows that $Filt(\cX)$ is the smallest pseudo-torsion class containing $\cX$.
\end{proof}

\subsection{Pseudo-wide subcategories}\label{ss13: wide subcats}

\begin{defn}\label{def: pseudo-wide} By a \textbf{pseudo-wide subcategory} of $\cG$ we mean a nonempty full subcategory $\cW$ of the strict category $\cG$ (with only strict morphisms) having the following properties, keeping Proposition \ref{prop: ker and coker in G} in mind:
\begin{enumerate}
\item Given any strict morphism $f:A\to B$ where $A,B\in\cW$, the kernel, image, and cokernel of $f$ lie in $\cW$. In particular, $0\in \cW$.
\item $\cW$ is closed under strict extensions in $\cG$.
\end{enumerate}
\end{defn}

By Proposition \ref{prop: ker and coker in G} and the definition of strict morphisms, $\cG$ is a pseudo-wide subcategory of itself. Although pseudo-wide subcategories have many of the properties of abelian categories, Example \ref{eg: A3 example} shows that they may not be additive categories.

\commentout{
\subsection{Relative pseudo-torsion classes}\label{ss14: relative p-tor}

\begin{defn}\label{def: relative ptor}
Let $\cX$ be pseudo-wide subcategory of $\cG$. Then a \textbf{relative pseudo-torsion class} in $\cX$ is a subset $\cP_0$ of $\cX$ which is closed under strict extensions and under strict quotients which lie in $\cX$. Thus, if $P\onto X$ is a strict epimorphism with $P\in\cP_0$ and $X\in \cX$ then $X\in\cP_0$. \textbf{Relative pseudo-torsionfree classes} in $\cX$ are defined analogously. 
\end{defn}

The \textbf{relative perpendicular classes} are defined to be $\cP_0^\perp\cap \cX$ and $\,^\perp \cQ_0\cap \cX$. Thus, for instance, the relative right perpendicular class of $\cP_0$ is the class of all $X\in \cX$ so that there are no nonzero strict morphisms from any $P\in\cP_0$ to $X$.

\begin{thm}\label{thm: rel p-tor pairs}
    Taking relative perpendicular classes gives a bijection, i.e., the left relative perpendicular class of the right relative perpendicular class of $\cP_0$ is the original class $\cP_0$ and similarly starting with $\cQ_0$. Furthermore, every object of $\cX$ has a canonical strict subobject in $\cP_0$ with quotient in $\cP_0^\perp\cap \cX$.
\end{thm}

The proof is very similar to the proof of the absolute cases: Theorems \ref{thm: duality between P,Q} and \ref{thm: tM and fM}. We will not need these in Part I (the present paper).
}

\section{Wall and chamber structure}\label{new sec2}

Recall that $\Lambda$ is a finite dimensional algebra over a field of rank $n$. Up to isomorphism, there are $n$ indecomposable projective modules $P_1,\cdots,P_n$ whose tops $S_i=P_i/rad\,P_i$ are the simple $\Lambda$-modules. Let $F_i$ denote the division algebra
\[
F_i:=\End_\Lambda(P_i)/rad\End_\Lambda(P_i)
\]
For any finitely generated right $\Lambda$-module $M$ let $\undim M$ denote its \textbf{dimension vector} 
\[
	\undim M=(d_1,\cdots,d_n)
\]
where $d_i$ is the dimension of $\Hom_\Lambda(P_i,M)$ over $F_i$. This is also the number of times $S_i$ occurs in the composition series of $M$. Then
\[
	\undim M\in K_0\Lambda\cong \ZZ^n.
\]
\begin{defn}\label{def: stability space}
We define the \textbf{stability space} of $\Lambda$ to be
\[
	V_\Lambda:=\Hom_\ZZ(K_0\Lambda,\RR)\cong \RR^n.
\]
Elements of $V_\Lambda$ are called \textbf{stability conditions}. For any stability condition $\theta\in V_\Lambda$ and $M\in mod\text-\Lambda$, let $\theta(M)$ denote
\[
	\theta(M):=\theta(\undim M)\in\RR.
\]
\end{defn}

Basic examples of pseudo-torsion classes, pseudo-torsionfree classes and pseudo-wide subcategories are given by ``stability conditions''. We could also take the ``reduced'' stability space of $\cG$ which is
\[
    V_\cG:=\Hom(\overline K_0\cG,\RR).
\]
This would be more natural and, in general, will give more pseudo-torsion classes for $\cG$. See Section \ref{new sec4} for explanations and examples. However, we study in detail stability diagrams for $\cG$ in the standard stability space $V_\Lambda$ since this will be more familiar to the readers. Also, we will be able to compare the stability diagrams of $\cG$ and $mod\text-\Lambda$ as subsets of the same space $V_\Lambda$. (On the other hand, $\cG\subseteq mod\text-\Lambda$, so $V_\Lambda\subseteq V_\cG$.)

\begin{defn}\label{def: extended dim M}
    For any object $M\in\cG$, we define the \textbf{extended dimension vector} of $M$ to be the collection of all dimension vectors of all strict subquotients of $M$. More precisely, it is the set of all ordered triples $([A],[B],[C])$ of elements of $K_0\Lambda=\ZZ^n$ where $(A,B,C)$ gives the subquotients of a strict filtration of $M$:
    \[
    M_1\cof M_2\cof M
    \]
    with $A=M_1$, $B=M_2/M_1$, $C=M/M_2$.
\end{defn}

Note that, although there may be infinitely many $M\in\cG$ and infinitely many strict filtrations of each $M$, there are only finitely many extended dimension vectors of bounded total length. Consequently, the set of extended dimension vectors of all objects of $\cG$ is only countably infinite.

\subsection{Walls thick and thin}\label{ss21: walls}

{Let $\cW(\cG)$ be the set of all $\theta\in V_\Lambda$ so that $\cW(\theta)\neq0$. We will see that this is the union of ``pseudo-walls'' $D_\cG(X)$. The pseudo-chambers will be the path components of the complement of $\cW(\cG)$ in $V_\Lambda$. {\color{black}We also discuss ``infinitesimal chambers''} which arise when a wall is a limit of other walls. In general, each wall will have two infinitesimal chambers, one on each side of the wall. We will show that $\cP(\theta)$ is constant on each pseudo-chamber and distinct for distinct pseudo-chambers. The same holds for $\cQ(\theta)$ since $\cP(\theta)$ and $\cQ(\theta)$ determine each other by Theorem \ref{thm: duality between P(theta) and Q(theta)}.

\color{black}
\begin{defn}\label{def: pseudo-wall}%
Given an object $M$ in a torsion class $\cG$ we define the \textbf{pseudo-wall} $D_\cG(M)$ to be the set of all stability conditions $\theta$ so that $\theta(M)=0$ and $\theta(M')\le0$ for all strict subobjects $M'$ of $M$. \commentout{{\color{orange} why consider all $\theta$ for the whole algebra? Is there a way to define the stability space for the strict category $\cG$?}{\color{blue} yes. we could define $V_\cG$ which would be the stability diagram for $\cG$, not the stability diagram. It is just that this is what people are used to. Also, the ghosts are walls in the stability diagram of $\Lambda$. Maybe we can put an example in Section 4 where we consider $\cG$ as wide subcategory.}} Equivalently, $\theta(M'')\ge0$ for all strict quotients $M''$ of $M$. The \textbf{interior} $int\,D_\cG(M)$ is the set of all $\theta\in D_\cG(M)$ so that $\theta(M')<0$ for all nonzero proper strict subobjects $M'\subset M$. The \textbf{boundary} of $D_\cG(M)$ is the complement of its interior:
\[
	\partial D_\cG(M):=D_\cG(M)-int\,D_\cG(M).
\]

For any $\theta\in V_\Lambda$, $\cW(\theta)$ is defined to be the class of all objects $X$ in $\cG$ so that $\theta \in D_{\cG}(X)$. That is, the class of all $X$ in $\cG$ such that $\theta(X)=0$ and $\theta(X')\le 0$ for all strict subobjects $X'$ of $X$. Equivalently, $\theta(X'')\ge0$ for all strict quotients $X''$ of $X$.
\end{defn}
}
By its definition, $D_\cG(M)$ is a closed convex subset of the stability space $V_\Lambda$. We also note that the set $D_\cG(M)\subset V_\Lambda$ depends only on the extended dimension vector of $M$ (Def. \ref{def: extended dim M}). So, there are only countably many pseudo-walls.

There is a big distinction between what we call ``thin'' walls and ``thick'' walls.

\begin{defn}\label{def: thin walls}
    The \textbf{thin walls} are the path components of the set of all $\theta\in\cW(\cG)$ so that $\cW(\theta)$ is generated by one object $B$. Thus, the objects of $\cW(\theta)$ will be the iterated strict self-extensions of $B$. Points on the rest of $\cW(\cG)$ will be called \textbf{thick}. The \textbf{thick walls} are $D_\cG(M)$ for $M \in \cG$ minus the thin walls.
\end{defn}

\begin{prop}\label{prop: W(theta) constant on thin wall}
    For any path $\theta_t$, $t\in[0,1]$ in a thin wall, $\cW(\theta_t)$ is constant, i.e., the generator $B_t$ of $\cW(\theta_t)$ is the same for all $0\le t\le1$.
\end{prop}

\begin{proof} Whether an object $M$ lies in $\cW(\theta_t)$ depends only on the extended dimension vector of $M$. So, if this lies on a thin wall, the generator $B_t$ is uniquely determined by its extended dimension vector. (Any other object $B$ with the same extended dimension vector as $B_t$ must lie in $\cW(\theta_t)$ and therefore must be isomorphic to $B_t$.)
    The set of all $t\in[0,1]$ with a fixed generator $B$ is a closed set since $D_\cG(B)$ is a closed set. If the generator $B_t$ of $\cW(\theta_t)$ is not a constant function of $t$, we would then get the compact connected set $[0,1]$ being a disjoint union of countably many closed subsets which is not possible by \cite{Sierpinski}.
\end{proof}

{\begin{prop}\label{prop: dD in int}%
Every $\theta\in \partial D_\cG(M)$ lies in $int\,D_\cG(\overline M)$ for some strict quotient $\overline M$ of $M$. $\theta$ also lies in $int\,D_\cG(M')$ for some strict subobject $M'$ of $M$.
\end{prop}

\begin{proof}%
Suppose $\theta\in \partial D_\cG(M)$. Then $\theta(M)=0$ and $\theta(M')=0$ for some strict subobject of $M$. Take $M'$ minimal. Then $\theta(M'')>0$ for every strict subobject $M''$ of $M'$. Therefore, $\theta\in int\,D_\cG(M')$. The corresponding statement holds for strict quotients of $M$ by transitivity of strict quotients.
\end{proof}

\begin{cor}\label{cor: D in int}%
The union of all the walls $D_\cG(M)$ is equal to the union of all their interiors $int\,D_\cG(M)$: $$\bigcup_{M\in\cG} D_{\cG}(M) = \bigcup_{M\in\cG} int D_{\cG}(M).$$
\end{cor}
}

{%
A nonzero object $B\in \cG$ is called a \textbf{pseudo-brick} if every nonzero strict endomorphism $f:B\to B$ is an isomorphism. Note that, if $M$ is not a brick and $f:M\to M$ is a nonzero strict endomorphism with kernel $K$, image $I$ and cokernel $C$ then, for any $\theta\in D_\cG(M)$ we have $\theta(K)=\theta(I)=\theta(C)=0$. So, the interior of $D_\cG(M)$ would be empty. This shows the following.

\begin{cor}\label{cor: only need bricks}%
The union of the walls $D_\cG(M)$ is equal to the union of all interiors $int\,D_\cG(B)$ for all pseudo-bricks $B$: $$\bigcup_{M\in\cG} D_{\cG}(M) = \bigcup_{\text{bricks }B\in\cG} int D_{\cG}(B).$$
\end{cor}
}

\subsection{Pseudo-torsion classes given by stability conditions}\label{ss22: chambers}

We are now ready to show how {pseudo-torsion classes and pseudo-torsionfree classes} arise from stability conditions.

{%

\begin{defn}\label{def: P,Q,W}
For any $\theta\in V_\Lambda$ we define $\cP(\theta),\overline\cP(\theta),\cQ(\theta),\overline\cQ(\theta)$ as follows.
\begin{enumerate}
    \item $\cP(\theta)$ is the class consisting of $0$ and all nonzero objects $M$ in $\cG$ so that $\theta(M)>0$ and $\theta(M'')>0$ for all nonzero strict quotients $M''$ of $M$.
   \item $\overline\cP(\theta)$ is the class of all objects $M$ in $\cG$ so that $\theta(M)\ge0$ and $\theta(M'')\ge0$ for all strict quotients $M''$ of $M$.
    \item $\cQ(\theta)$ is the class consisting of $0$ and all nonzero objects $N$ in $\cG$ so that $\theta(N)<0$ and $\theta(N')<0$ for all nonzero strict subobjects $N'$ of $N$.
    \item $\overline\cQ(\theta)$ is the class of all objects $N$ in $\cG$ so that $\theta(N)\le0$ and $\theta(N')\le0$ for all strict subobjects $N'$ of $N$.
\end{enumerate}
Note that $0$ is an object in each of these classes.
\end{defn}
}

\ki{
\begin{lem}\label{lem: W=Pbar cap Qbar}
For any $\theta\in V_\Lambda$ we have: 
\[
\cW(\theta)=\overline\cP(\theta)\cap \overline\cQ(\theta).
\]
\end{lem}

\begin{proof}
This follows directly from the definitions of the three terms.
\end{proof}
}

{
\begin{lem}\label{lem: Pbar=P when W=0}
    When $\cW(\theta)=0$ \commentout{{\color{orange} (no modules are $\theta$ semi-stable?) \color{blue} Yes!}} we have $\overline \cP(\theta)=\cP(\theta)$ and $\overline\cQ(\theta)=\cQ(\theta)$. When $\cW(\theta)$ is nonzero, we have strict inclusions $ \cP(\theta)\subsetneq \overline\cP(\theta)$ and $ \cQ(\theta)\subsetneq \overline\cQ(\theta)$.
\end{lem}

\begin{proof}
Clearly, $\cP(\theta)\subseteq \overline \cP(\theta)$. When $\cW(\theta)\neq0$, the inclusion is strict since $\cW(\theta)$ is contained in $\overline\cP(\theta)$ but disjoint from $\cP(\theta)$. When $\cW(\theta)=0$, we have equality. To see this, take any $X\in \overline \cP(\theta)$. Then $\theta(X')\ge0$ for all strict quotients $X'$ of $X$. If $\theta(X')=0$ then we would have $X'\in \cW(\theta)$. So, we must have $\theta(X')>0$ for all $X'$ strict quotients of $X$ including $X'=X$. So, $X\in\cP(\theta)$ and we have equality: $\overline\cP(\theta)=\cP(\theta)$. The argument for $\overline \cQ(\theta)$ is similar.
\end{proof}
}

{%
\begin{thm}\label{thm: P(theta) is P}
    For every $\theta\in V_\Lambda$, $\cP(\theta)$ and $\overline\cP(\theta)$ are pseudo-torsion classes which are closed under extensions. In particular, they are closed under direct sums.
\end{thm}

\begin{proof}
    Let $M\in \cP(\theta)$ and let $M''$ be a strict quotient of $M$. Then $\theta(M'')>0$ and $\theta(N)>0$ for any strict quotient $N$ of $M''$ since $N$ is also a strict quotient of $M$. So $M''\in\cP(\theta)$ and $\cP(\theta)$ is closed under strict quotients.

    Next, we show that $\cP(\theta)$ is closed under extensions, in particular strict extensions. We let $0\to A\to B\to C\to0$ be a short exact sequence where $A,C\in\cP(\theta)$. Then $B\in\cG$ since $\cG$ is closed under extensions and
    \[
    \theta(B)=\theta(A)+\theta(C)>0.
    \]
    By Proposition \ref{prop: duality of strict ses}, any strict quotient $B''$ of $B$ is an extension of a strict quotient $A''$ of $A$ by a strict quotient $C''$ of $C$. Then 
        \[
    \theta(B'')=\theta(A'')+\theta(C'')>0.
    \]
    So, $B\in \cP(\theta)$ as claimed. The proof for $\overline\cP(\theta)$ is similar.
\end{proof}

An analogous argument, with inequalities reversed, proves the following.
\begin{thm}\label{thm: Q(theta) is Q}
    For any $\theta\in V_\Lambda$, $\cQ(\theta)$ and $\overline\cQ(\theta)$ are pseudo-torsionfree classes which are closed under extensions.
\end{thm}
}

{%
\begin{thm}\label{thm: W(theta) is W}
    For any $\theta\in V_\Lambda$, $\cW(\theta)$ is a pseudo-wide subcategory of $\cG$ which is closed under extensions.
\end{thm}

\begin{proof}
    Let $X,Y\in \cW(\theta)$ and let $f:X\to Y$ be a strict morphism with image $I$, kernel $K$ and cokernel $C$. Then $\theta(X)=0=\theta(Y)$ and $\theta(I)\le0$ being a strict subobject of $Y$ and $\theta(I)\ge0$ being a strict quotient of $X$. So, $\theta(I)=0$. This makes $\theta(K)=\theta(X)-\theta(I)=0$ and $\theta(C)=\theta(Y)-\theta(I)=0$.
    
    Any strict subobject $K'$ of $K$ is also a strict subobject of $X$. So $\theta(K')\le0$ and $K\in\cW(\theta)$. Similarly, $I\in\cW(\theta)$ since any strict subobject of $I$ is also a strict subobject of $Y$. Finally, $C\in \cW(\theta)$ since any strict quotient $C''$ of $C$ is a strict quotient of $Y$ making $\theta(C'')\ge0$.

\ki{Since $\overline\cP(\theta)$ and $\overline\cQ(\theta)$ are closed under extensions by Theorems \ref{thm: P(theta) is P} and \ref{thm: Q(theta) is Q}, so is their intersection $\cW(\theta)=\overline\cP(\theta)\cap\overline\cQ(\theta)$.} 
\end{proof}
}

\begin{rem}\label{rem: not extension closed}
    Not all pseudo-torsion classes are closed under extension. For example, in the $B_2$ example below (\ref{B2 example}), the class of objects $\cP=add(P_1\oplus I_1)$ is not closed under extension of $\Lambda$-modules since $0\to P_1\to I_2\oplus I_2\to I_1\to 0$ is an extension and $I_2\oplus I_2$ is not in $\cP$. However, $\cP$ is a pseudo-torsion class since there are no strict morphisms from $P_1$ or $I_1$ into $mI_2$ for any $m\ge1$. So, $\cP$ is closed under strict extensions and also under strict quotients since \ki{the only strict quotients of $nP_1\oplus mI_1$ are multiples of $P_1$. (See Figure \ref{Fig: 3 circles} where $\cP$ is $\cP_6$.)}

    The fact that $\cP(\theta)$ is closed under $\Lambda$-extensions is a consequence of the fact that this is defined for $\theta$ in the stability space of $\Lambda$. If we look at the stability space of $\cG$ we will see two more pseudo-chambers giving two new pseudo-torsion classes. The first is \ki{$\cP_6=add(P_1\oplus I_1)$ discussed above and the second is $\cP_7=\,^\perp \{P_1,I_1\}$, the collection of all objects $mI_2\oplus kI_1$ where $m>0$.} This is closed under strict quotients since the epimorphism $2I_2\to I_1$ is not a strict morphism.
\end{rem}

{%
\subsection{Perpendicular classes}\label{ss: perp classes}
We will see that $\cP(\theta)$ and $\cQ(\theta)$ are perpendicular to each other according to the following definition if $\cW(\theta)=0$. When $\cW(\theta)\neq0$, we will see that $\cP(\theta)^\perp=\overline \cQ(\theta)$ and $\overline\cP(\theta)^\perp= \cQ(\theta)$. Recall from Definition \ref{def: perp class} that we only consider strict morphisms when referring to perpendicular classes.

\begin{thm}\label{thm: duality between P(theta) and Q(theta)}
    If $\cW(\theta)=0$ then $\cP(\theta)$ is the left perpendicular class of $\cQ(\theta)$ and (equivalently) $\cQ(\theta)=\cP(\theta)^\perp$.
\end{thm}

\begin{proof}
This follows from Theorem \ref{thm: duality between P(theta) and Q(theta)bar} below since $\cW(\theta)=0$ implies $\overline\cQ(\theta)=\cQ(\theta)$ by Lemma \ref{lem: Pbar=P when W=0}.
\end{proof}

\begin{thm}\label{thm: duality between P(theta) and Q(theta)bar}
    For any $\theta\in V_\Lambda$, $\cP(\theta)$ is the left perpendicular class of $\overline\cQ(\theta)$ and (equivalently) $\overline\cQ(\theta)=\cP(\theta)^\perp$.
\end{thm}

\begin{proof}
    If $M\in \cP(\theta)$ and $N\in\overline\cQ(\theta)$ then there cannot be a nonzero strict morphism $f:M\to N$ since its image $\im f$ lies in both $\cP(\theta)$ and $\overline\cQ(\theta)$ making $\theta(\im f)>0$ and $\theta(\im f)\le0$ which is a contradiction unless $\im f=0$. Thus $\overline\cQ(\theta)\subset \cP(\theta)^\perp$.

    Conversely, suppose that $Y$ is not in $\overline\cQ(\theta)$. Then there is a strict subobject $Y'$ of $Y$ so that $\theta(Y')>0$. Take $Y'$ minimal. Then $\theta(Z)\le 0$ for all nonzero strict subobjects of $Y$ which are properly contained in $Y'$. Then
    \[
    	\theta(Y'/Z)=\theta(Y')-\theta(Z)>0.
    \]
This implies that $Y'\in\cP(\theta)$ and the inclusion map $Y'\cof Y$ is a nonzero strict morphism. So, $Y$ is not in $\cP(\theta)^\perp$. We conclude that $\cP(\theta)^\perp=\overline\cQ(\theta)$.
\end{proof}
}

{
By a similar argument we have the following.
\begin{thm}\label{thm: duality between P(theta)bar and Q(theta)}
    For any $\theta\in V_\Lambda$, $\overline\cP(\theta)$ is the left perpendicular class of $\cQ(\theta)$ and (equivalently) $\cQ(\theta)=\overline\cP(\theta)^\perp$.
\end{thm}
}

The statements that $\overline\cQ(\theta)=\cP(\theta)^\perp$ and $\cQ(\theta)=\overline\cP(\theta)^\perp$ follow from similar arguments. They also follow from the more general formula in {Theorem \ref{thm: duality between P,Q}.}

\subsection{pseudo-chambers}\label{ss: pseudo-chambers}

{%
Recall that $\cW(\cG)$ was defined to be the set of all $\theta\in V_\Lambda$ so that $\cW(\theta)\neq0$. This is the same as the union of all pseudo-walls $D_\cG(M)$ {for $M \in \cG$} because if $\cW(\theta)\neq 0$, it lies in the wall $D_{\cG}(M)$ for some $M\in\cG$. The standard definition of a ``pseudo-chamber'' is a component of the complement of the closure of $\cW(\cG)$ in $V_\Lambda$. By a \textbf{pseudo-chamber} we mean a path component in the complement of $\cW(\cG)$. We will show that each pseudo-chamber has a well-defined pseudo-torsion class {and a well-defined pseudo-torsionfree class} and that these are different for different pseudo-chambers.}

{
\begin{defn}\label{defn: U(M)}%
For every object $M$ of $\cG$, let $\cU(M)$ denote the set of all $\theta$ so that $M\in \cP(\theta)$. Let $\overline\cU(M)$ denote the set of all $\theta$ so that $M\in \overline\cP(\theta)$. {}Let $\cV(M)$ denote the set of all $\theta$ so that $M\in \cQ(\theta)$. Similarly, $\overline\cV(M)$ is the set of all $\theta$ so that $M\in \overline\cQ(\theta)$.
\end{defn}

Clearly, $\cU(M),\cV(M)$ are open, $\overline\cU(M),\overline\cV(M)$ are closed and all four sets are convex.}

{
\begin{lem}\label{lem: dU(M) is union of D(M')}%
    a) $\overline\cU(M)$ is the closure of $\cU(M)$.

{b) {}$\overline\cV(M)$ is the closure of $\cV(M)$.}

    c) $\partial\cU(M)=\overline\cU(M)-\cU(M)$ is contained in the union of walls $D_\cG(\overline M)$ where $\overline M$ is a strict quotient of $M$ (including $\overline M=M$).

   d) {}$\partial\cV(M)=\overline\cV(M)-\cV(M)$ is contained in the union of walls $D_\cG(M')$ where $M'$ is a strict subobject of $M$ (including $M'=M$).
\end{lem}

\begin{proof}%
    a) For any $\theta\in \overline\cU(M)$ we have $\theta(\overline M)\ge0$ for all strict quotients $\overline M$ of $M$. So $\theta+\varepsilon\eta\in \cU(M)$ for all $\varepsilon>0$. These converge to $\theta$ so, $\theta$ is in the closure of $\cU(M)$.

    c) If $\theta\in\partial \cU(M)$ then $\theta(\overline M)\ge0$ for all strict quotients $\overline M$ of $M$ and equality holds for at least one such $\overline M$. Taking $\overline M$ to be minimal we see that $\theta\in D_\cG(\overline M)$.
    
    (b) and (d) are proved analogously.
\end{proof}
}

{

\begin{thm}\label{thm: P(theta) is constant on path components}%
$\cP(\theta)$ is constant on path components of the complement of $\cW(\cG)$.
\end{thm}
\begin{proof}%
    Let $\theta_t$, $0\le t\le 1$, be a continuous path in the complement of $\cW(\cG)$. If $M\in\cP(\theta_0)$ and $M\notin\cP(\theta_1)$ then $\theta_0\in \cU(M)$ and $\theta_1\notin\cU(M)$. So, the path $\theta_t$ must pass through $\partial\cU(M)$ which is a union of walls, a contradiction. So, $\cP(\theta_0)=\cP(\theta_1)$ as claimed.
\end{proof}
}

{
\begin{thm}\label{thm: P(theta) are different in different path components}%
    If $\theta_0,\theta_1$ lie in distinct path components of the complement of $\cW(\cG)$ then $\cP(\theta_0)\neq\cP(\theta_1)$.
\end{thm}

\begin{proof}%
    Take the linear path $\theta_t=t\theta_1+(1-t)\theta_0$. Since $\theta_0,\theta_1$ are in different path components, the path crosses some wall, say $\theta_{t_0}\in D_\cG(X)$. Take $X$ to be minimal. Then $\theta_{t_0}\in int\,D_\cG(X)$. There are two cases. Either the path is parallel to this wall or it crosses the wall transversely. 
    
    In the second case we may assume, by symmetry, that $\theta_0(X)<0$ and $\theta_1(X)>0$. So, $X\notin\cP(\theta_0)$. We claim that $X\in\cP(\theta_1)$. Suppose not. Then $\theta_1(X')\le 0$ for some nonzero strict quotient $X'$ of $X$. Take $X'$ minimal. Then $\theta_1(X'')>0$ for all strict quotients $X''$ of $X$. But $\theta_{t_0}(X')>0$ since $\theta_{t_0}\in int\,D_\cG(X)$. Since the path $\theta_t$ is linear, there is some $t_1\in (t_0,1]$ so that $\theta_{t_1}(X')=0$ and $\theta_{t_1}(X'')>0$ for all proper strict quotients $X''$ of $X'$. Then $\theta_{t_1}\in D(X')$ contradicting the minimality of $X$. 
    
    Suppose now that the path is parallel to the wall $D_\cG(X)$. Then we have $\theta_0(X)=0=\theta_1(X)$. Since $\theta_1$ does not lie on any wall, there must be some strict subobject $X'$ of $X$ so that $\theta_1(X')>0$. Take $X'$ minimal. Since $\theta_t(X')<0$, there must be some $t<s<1$ so that $\theta_s(X')=0$. By minimality of $X'$ we have $\theta_1(Z)\le0$ for any $Z\subset X'$. So, $\theta_s(Z)\le0$ giving $X'\in \cW(\theta_s)$ contradicting the minimality of $X$. 

    Since $\theta_1(X')>0$ and $\theta_1(Z)\le0$ for all strict $Z\subset X'$, we have $\theta_1(X'/Z)>0$ for all such $Z$ making $X'\in\cP(\theta_1)$. But $\theta_0(X')<0$. So, $X'\notin \cP(\theta_0)$. So $\cP(\theta_0)\neq \cP(\theta_1)$.
\end{proof}
}

{
\begin{rem}\label{rem: path components are convex}%
The proof of Theorem \ref{thm: P(theta) are different in different path components} shows that the path components of the complement of $\cW(\cG)$ are convex. Otherwise the straight line connecting two points would pass through some walls making the value of $\cP(\theta_t)$ different at the two endpoints contradicting Theorem \ref{thm: P(theta) is constant on path components}.
\end{rem}
}

{
Remark \ref{rem: path components are convex} can be formalized into the following theorem.
\begin{thm}\label{thm: when P(theta) is constant}
Given $\theta_0,\theta_1\in V_\Lambda$, $\cP(\theta_0)=\cP(\theta_1)$ if and only if the following are satisfied.
\begin{enumerate}
	\item The linear path $\theta_t=t\theta_1+(1-t)\theta_0$ does not cross any walls transversely in its interior. By this we mean that, if $\theta_t\in D_\cG(M)$ for some $0<t<1$ then the entire path $\theta_t$ lies in the hyperplane perpendicular to $\undim M$. I.e., $\theta_t(M)=0$ for all $t\in [0,1]$.
	\item If $\theta_0\in D_\cG(L)$ for some $L\in\cG$ then
	\[
		\frac d{dt}\theta_t(L)|_{t=0}\le0.
	\]
	\item If $\theta_1\in D_\cG(R)$ for some $R\in\cG$ then
	\[
		\frac d{dt}\theta_t(R)|_{t=1}\ge0.
	\]
\end{enumerate}
\end{thm}
}

{\begin{proof}
    As observed in Remark \ref{rem: path components are convex}, if the path $\theta_t$ does not cross any walls transversely, then $\cP(\theta_t)$ is constant for $0<t<1$. For $t=0$, an object $L$ would be added to $\theta_0$ if the wall $D_\cG(L)$ was oriented in the other direction (so that $\frac d{dt}\theta_t(L)|_{t=0}>0$). This not being the case, we have $\cP(\theta_0)=\cP(\theta_{1/2})$. Similarly for the case $t=1$. So, $\cP(\theta_t)$ is constant for $t\in[0,1]$ assuming $(1),(2),(3)$.

    Conversely, suppose that the path $\theta_t$ crosses some wall $D_\cG(X)$ at time $t=t_0$. Then $X$ or some strict quotient of $X$ would be added to $\cP(\theta_t)$ for $t>t_0$ (or for $t<t_0$) depending on the sign of $\frac d{dt}\theta_t(X)|_{t=t_0}$. So, $\cP(\theta_0)\neq \cP(\theta_1)$ if any of the conditions $(1),(2),(3)$ were violated.
\end{proof}
}

Similarly we have the following.

{
\begin{thm}\label{thm: when Q(theta) is constant}
Given $\theta_0,\theta_1\in V_\Lambda$, $\cQ(\theta_0)=\cQ(\theta_1)$ if and only if the following are satisfied.
\begin{enumerate}
	\item The linear path $\theta_t=t\theta_1+(1-t)\theta_0$ does not cross any walls transversely in its interior. 
	\item If $\theta_0\in D_\cG(L)$ for some $L\in\cG$ then
	\[
		\frac d{dt}\theta_t(L)|_{t=0}\ge0.
	\]
	\item If $\theta_1\in D_\cG(R)$ for some $R \in \cG$ then
	\[
		\frac d{dt}\theta_t(R)|_{t=1}\le0.
	\]
\end{enumerate}
\end{thm}
}

\begin{cor}\label{cor: P,Q constant on thin walls}
    For $\theta$ on any thin wall, $\cP(\theta)$ and $\cQ(\theta)$ are constant.
\end{cor}

\begin{cor}\label{cor: P,Q equal on convex sets}
    Given any pseudo-torsion class $\cP$, the set $S(\cP)$ of all $\theta\in V_\Lambda$ so that $\cP(\theta)=\cP$ is convex. Similarly, for any pseudo-torsionfree class $\cQ$, the set $T(\cQ)$ of all $\theta\in V_\Lambda$ so that $\cQ(\theta)=\cQ$ is convex.
\end{cor}

\begin{proof} Let $\theta_0,\theta_1\in S(\cP)$. Then Theorem \ref{thm: when P(theta) is constant} implies that, for all $t\in[0,1]$, $\cP(\theta_t)=\cP$. Therefore, $S(\cP)$ is convex. Similarly, $T(\cQ)$ is convex by Theorem \ref{thm: when Q(theta) is constant}.
\end{proof}

\subsection{Examples}\label{ss1.3:examples}

We compute the pseudo-wall and chamber structure for two examples of torsion classes and compare them with the full wall and chamber structures given all modules: (a) a torsion class for the modulated quiver of type $B_2/C_2$. (b) a torsion class for a quiver of type $A_3$.

In the following examples and later in the text we use the following definitions.

\begin{defn}
    Let $\cX$ be any collection of objects in $\cG$. By the pseudo-torsion class \textbf{generated} by $\cX$ we mean the intersection of all pseudo-torsion classes containing $\cX$. \ki{Similarly, the pseudo-torsionfree class {\bf generated} by $\cY\subset\cG$ is the intersection of all pseudo-torsionfree classes containing $\cY$. We also say that the pseudo-torsion class $\,^\perp\cY$ is {\bf cogenerated} by $\cY$ and $\cX^\perp$ is {\bf cogenerated} by $\cX$.}
\end{defn}

By Theorem \ref{thm: duality between P,Q}, the pseudo-torsion class generated by $\cX$ is given by $\,^\perp(\cX^\perp)$ and by Proposition \ref{prop: filt X}, it is also $Filt(\cX)$. \ki{Similarly, the pseudo-torsionfree class generated by $\cY$ is $(\,^\perp \cY)^\perp$.}

\ki{In the following examples, we need the following.

\begin{prop}\label{prop: when A+B is G-indecomposable}
Suppose that $A,B\in\cG$ and there is a morphism $f:A\to B$ with kernel $X\notin\cG$. Then $A\oplus B$ is indecomposable in the strict category $\cG$. Furthermore, if $B$ is a simple $\Lambda$-module, we have a strict extension: $B\cof A\oplus B\onto A$.
\end{prop}

\begin{proof}
$A$ is not a strict subobject of $A\oplus B$ since its intersection of the graph of $f$ with $X\oplus 0$ which is not in $\cG$. Thus, the projection $A\oplus B\to B$ is not strict. However, any simple subobject of any $M\in\cG$ is a strict subobject. So, $B\cof A\oplus B$ if $B$ is simple.
\end{proof}
}

\subsubsection{$B_2$ example}\label{B2 example}
{The first example is the modulated quiver $\CC\to\RR$. The projective modules are $P_2:0\to\RR$ and $P_1:\CC\to \RR^2$. The injective modules are $I_1:\CC\to0$, $I_2:\CC\to\RR$. The Auslander-Reiten quiver is:
\[
\xymatrix{
& \bf P_1\ar[dr]\ar@{--}[rr] && \bf I_1\\
P_2\ar[ur]\ar@{--}[rr]&& \bf I_2\ar[ur]
	} 
\]
We consider the torsion class $\cG=\prescript{\perp}{}P_2=\{P_1,I_2,I_1\}$. In the strict category, we claim that there are no nonzero strict morphisms between the objects $P_1,I_2,I_1$. The proof is similar to the calculation in {Example \ref{eg: A3 example}}. Consider the example of the main diagram:
\[
\xymatrixrowsep{12pt}\xymatrixcolsep{12pt}
\xymatrix{
& 0\ar[d] & 0\ar[d] & 0\ar[d]\\
0\ar[r]& S_2\ar[d]\ar[r] &
	I_2\ar[d]\ar[r] &
	I_1\ar[d]^=\ar[r]
	& 0\\
0\ar[r]& P_1\ar[d]\ar[r] &
	I_2\oplus I_2\ar[d]\ar[r] &
	I_1\ar[d]\ar[r]
	& 0\\
0\ar[r]& I_2\ar[d]\ar[r]^= &
	I_2\ar[d]\ar[r] &
	0\ar[d]\ar[r]
	& 0\\
& 0 & 0 & 0	} 
\]
where the middle row is an almost split sequence, the middle column is a split exact sequence, and the other objects are uniquely determined. For example, in the upper left corner we have the kernel of $I_2\to I_2\oplus I_2\to I_1$ which is $S_2$. Since $S_2$ is not in the torsion class, neither $P_1$ nor $I_2$ is a strict subobject of $I_2\oplus I_2$. Consequently, $I_2\oplus I_2$ has no strict quotients. Similarly, $mI_2$ has no nontrivial strict subobjects or strict quotient objects. So, $mI_2$ is indecomposable in the strict category. On the other hand, direct summands of $mP_1$ and $mI_1$ are strict subobjects since there are no other subobjects. 

The indecomposable objects $P_2,I_2,I_1$ have no strict quotients. This means that the walls $D_\cG(P_2),D_\cG(I_2),D_\cG(I_1)$ are hyperplanes, i.e., lines in $\RR^2$. The wall $D_\cG(m I_2)$ is also a hyperplane which, as a set, is equal to $D_\cG(I_2)$.
}

\ki{To construct the full stability diagram for the torsion class $\cG$, we need the list of all pseudo-bricks in $\cG$ which we obtained using Proposition \ref{prop: when A+B is G-indecomposable}.
\begin{enumerate}
    \item $P_1$, $I_2$, $I_1$ are pseudo-bricks since they are $\Lambda$-bricks, i.e., bricks in $mod\text-\Lambda$.
    \item $mI_2$ is a pseudo-brick since $(m-1)I_2$ is not a strict subobject of $mI_2$.
    \item $aP_1\oplus bI_2\oplus cI_1$ is a pseudo-brick if any two of the numbers $a,b,c$ are nonzero.
\end{enumerate}

We see that there are infinitely many pseudo-bricks and three open pseudo-chambers. See Figure \ref{fig: eg, B2}. The two open pseudo-chambers at the bottom have finitely generated pseudo-torsion classes:
\[
\cP_0=Filt(0)=\{0\},
\]
\[
\cP_1=Filt(I_1)=\{mI_1\}.
\]
The last open pseudo-chamber in Figure \ref{fig: eg, B2} has finitely cogenerated pseudo-torsion class
\[
\cP_3=\cG=\,^\perp 0.
\]
The other pseudo-chambers have no interior and their associated pseudo-torsion classes are neither finitely generated nor finitely cogenerated. For example, the point $\theta_0=(3,-2)$ shown in Figure \ref{fig: eg, B2} has pseudo-torsion class:
\[
	\cP(\theta_0)=\{aP_1\oplus bI_2\oplus cI_1\mid b>a\}.
\] 
This follows from a simple calculation on the strict quotient $aP_1\oplus bI_2$ of $aP_1\oplus bI_2\oplus cI_1$:
\[
\theta_0(aP_1\oplus bI_2)=3a+3b-4a-2b=b-a>0.
\]

For the applications, we need to ignore or delete the walls labeled by objects which are not $\Lambda$-indecomposable. We call the result the \textbf{reduced stability diagram} for $\cG$. This is the diagram where we keep only the three walls labeled with the three $\Lambda$-bricks $P_1,I_2,I_1$. The reduced stability diagram has six ``reduced'' chambers.} 

{
\begin{figure}[htbp]
\begin{center}
\begin{tikzpicture}[scale=1.5]
\begin{scope}[xshift=-5cm]
\draw[thick] (2,0)--(-2,0); 
\draw[thick] (0,1)--(0,-1); 
\draw[thick] (1,-1)--(0,0); 
\draw[thick] (2,-1)--(0,0); 
\draw (1,-1) node[below]{$D(I_2)$};
\draw (0,-1) node[below]{$D(I_1)$};
\draw (2,-1) node[below]{$D(P_1)$};

\coordinate (L0) at (-1,-.5);
\coordinate (L1) at (.3,-.8);
\coordinate (L2) at (1.2,-.8);
\coordinate (L3) at (1.3,-.3);
\coordinate (L5) at (-1.2,.8);
\coordinate (L4) at (1,.8);
\foreach \x in {L0,L1,L2,L3,L5}
\draw(\x)circle[radius=2mm];
\draw (L0)node{$\cG_0$};
\draw (L1)node{$\cG_1$};
\draw (L2)node{$\cG_2$};
\draw (L3)node{$\cG_3$};
\draw (L4)node{$mod\text-\Lambda$};
\draw (L5)node{$\cG_5$};
\end{scope}
%
\draw[thick] (0,1.2)--(0,-1.2); 
\draw[thick] (1,-1)--(-1,1); 
\draw[very thick,blue] (1.2,-1.2)--(-1.2,1.2); 
\draw[thick] (2,-1)--(-2.4,1.2); 
\draw[blue] (-1.2,1.2) node[above]{$D_\cG(mI_2)$};
\draw (0,1.2) node[above]{$D_\cG(I_1)$};
\draw (0,-1.2) node[below]{$D_\cG(I_1)$};
\draw (-2.5,1.2) node[above]{$D_\cG(P_1)$};
\draw (1.9,-.8) node[above]{$D_\cG(P_1)$};
\draw[blue] (1.2,-1.2) node[below]{$D_\cG(mI_2)$};

\coordinate (L0) at (-1,-.5);
\coordinate (L1) at (.3,-.8);
\coordinate (L2) at (1.6,-1.1);
\coordinate (L3) at (1,.5);
\coordinate (L5) at (-1.5,1.1);
\coordinate (L4) at (-.5,1.1);
\foreach \x in {L0,L1,L3}
\draw(\x)circle[radius=2mm];
\draw (L0)node{$\cP_0$};
\draw (L1)node{$\cP_1$};
\draw (L2)node{$\cP_2$};
\draw (L3)node{$\cP_3$};
\draw (L4)node{$\cP_4$};
\draw (L5)node{$\cP_5$};
\draw (-.9,.8)--(.9,-.8);
\draw (-1,.8)--(1,-.8);
\draw (-1.1,.8)--(1.1,-.8);
\draw (-1.2,.8)--(1.2,-.8);
\draw (1.1,-.8) node{$\bullet$};
\draw (1.2,-.75) node[below]{\tiny$\theta_0$};
\draw (-1.3,.8)--(1.3,-.8);
\draw (-1.4,.8)--(1.4,-.8);
\draw (-1.5,.8)--(1.5,-.8);
\draw (-.7,.8)--(0,0);
\draw (-.6,.8)--(0,0);
\draw (-.5,.8)--(0,0);
\draw (-.4,.8)--(0,0);
\draw (-.3,.8)--(0,0);
\draw (-.2,.8)--(0,0);
\draw (-.1,.8)--(0,0);
\end{tikzpicture}
\caption{The wall-and-chamber structure for $mod\text-\Lambda$ is shown on the left. On the right we have \ki{three open pseudo-chambers with pseudo-torsion classes $\cP_0,\cP_1,\cP_3$ equal to the torsion classes $\cG_0,\cG_1,\cG_3=\cG$ on the left. These are given by $\cG_0=\{0\}=\cP_0=Filt (0)$, 
$\cG_1=add(I_1)=\cP_1=Filt(I_1)$, $\cG_3=\cG=\cP_3=add(I_1,I_2,P_1)=\,^\perp0$. The other three regions, labeled $\cP_2,\cP_4,\cP_5$, are filled with infinitely many walls as shown. The associated pseudo-torsion classes including $\cP(\theta_0)$ are discussed in the text.}
}
\label{fig: eg, B2}
\end{center}
\end{figure}
}

\ki{
In Figure \ref{fig: eg, B2}, there are three regions in the reduced diagram which are cut up into infinitely many pseudo-chambers in the actual strict diagram. We associate to them the union of the pseudo-torsion classes for all points in the interior of those regions. These are easy to compute since the pseudo-torsion classes of $\theta$ are increasing as $\theta$ goes towards the positive region (the upper right corner). Thus 
\[
\cP_2=\cP(2,{-1})=\{aP_1\oplus bI_2\oplus cI_1\mid b>0\}= \prescript{\perp}{} P_1.
\]
\[
\cP_4=\cP(0,1)=\{aP_1\oplus bI_2\oplus cI_1\mid a+b>0\}= \prescript{\perp}{} I_1.
\]
$\cP_2$ and $\cP_4$ are finitely cogenerated since the points $(2,-1)$ and $(0,1)$ are not limit points from the right. (This is not a counterexample to Corollary \ref{cor: red boundary} since $\cP_2$ and $\cP_4$ are not pseudo-torsion classes of any pseudo-chambers.) Also,
\[
\cP_5= \cP({-1},1)=\{aP_1\oplus bI_2\oplus cI_1\mid a>c\}.
\]
$\cP_5$ is infinitely generated and infinitely cogenerated.
}

We call a pseudo-wall $\cX$ such as $D_\cG(mI_2)$ a \textbf{quasi-thin wall}. This means all the objects $M$ so that $D_\cG(M)=\cX$ are all $\Lambda$-direct sums of the same object $X$ even though the objects $mX$ are indecomposable in the strict category $\cG$.

{
\subsubsection{A torsion class in $A_3$}\label{eg: A3 example stability}
This example has rank 3. It is the torsion class $\,^\perp S_2$ for the quiver $1\to 2\to 3$. The Auslander--Reiten quiver, with torsion class $\cG$ circled is given as follows.
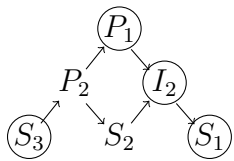
\begin{center}
\begin{tikzpicture}[scale=1.2]
{ 
\begin{scope}[xshift=4cm,yshift=-.5cm]
\draw (0,0) node{$S_3$}; 
\draw(0,0)circle[radius=2.4mm];
\draw (0.5,0.6) node{$P_2$}; 
\draw (1.5,0.6) node{$I_2$};
\draw(1.5,0.6)circle[radius=2.4mm];
\draw (1,1.2) node{$P_1$};
\draw(1,1.2)circle[radius=2.4mm];
\draw (1,0) node{$S_2$}; 
\draw (2,0) node{$S_1$}; 
\draw(2,0)circle[radius=2.4mm];
\draw[->] (0.13,0.15)--(.33,.4);
\draw[->] (1.13,0.15)--(1.33,0.4);
\draw[->] (0.63,0.75)--(.83,1);
\draw[->] (0.63,0.37)--(.83,.14);
\draw[->] (1.63,.37)--(1.83,.14);
\draw[->] (1.13,.97)--(1.33,.74);
\end{scope}
}
\end{tikzpicture}
\end{center}
Thus, there are four $\Lambda$-indecomposable objects $S_1$, $P_1$, $I_2$, $S_3$. \ki{The other pseudo-bricks in $\cG$ are given, using Proposition \ref{prop: when A+B is G-indecomposable}, as follows.
\begin{enumerate}
    \item $aP_1\oplus cS_1$ when $c>0$. This has strict subobjects $cS_1$ and $aS_3$.
    \item $bI_2\oplus cS_1$ when $c>0$. This has strict subobject $cS_1$.
    \item $aP_1\oplus bI_2\oplus cS_1$ when $c>0$. This has strict subobject $cS_1$ and $aS_3$ and the strict quotient $aP_1\oplus bI_2$ is a strict direct sum of its $\Lambda$-components.
\end{enumerate}
These are pseudo-bricks since there are non-admissible morphisms from $P_1$ and $I_2$ to $S_1$. So, for example, $aP_1$ is not a strict subobject of $aP_1\oplus cS_1$ when $c>0$ and there are no strict endomorphisms with kernel $aP_1$ or image $aP_1$. But $aS_1$ is a strict subobject in all three cases. So, the walls $D_\cG(aP_1\oplus bI_2\oplus cS_1)$ with $c>0$ are on the negative side of the hyperplane $D_\cG(S_1)$.
}

In Figure \ref{Fig: A3 example}, we draw the stereographic projection to $\RR^2$ of the intersection of walls with the unit sphere $S^2$ as in \cite{MoreGhosts}, \cite{IT14}. \ki{The shaded region is filled with a dense set of walls $D_\cG(bI_2\oplus cS_1)$. The walls $D_\cG(aP_1\oplus bI_2\oplus cS_1)$ are outside the $D_\cG(S_3)$ circle which represents the negative side of the hyperplane $D_\cG(S_3)$ since $S_3$ is a strict subobject of $aP_1\oplus bI_2\oplus cS_1$ when $a>0$.}

\ki{We label the three shaded regions with the pseudo-torsion classes which are the union of the pseudo-torsion classes in those regions. Using the points $a=(0,1,0)$, $b=(0,1,-1)$ and $c=(0,1,1)$ these are:
\[
    \cP_7=\cP(c)=\{aP_1\oplus bI_2\oplus cS_1\oplus dS_3\mid a+b+d>0\}
\]
\[
    \cP_8=\cP(a)=\{aP_1\oplus bI_2\oplus cS_1\mid a+b>0\}
\]
\[
    \cP_9=\cP(b)=\{bI_2\oplus cS_1\mid b>0\}
\]
For example, the coordinates of point $a$ are $(0,1,0)$ since $\cW(a)$ contains $S_1,S_3$. $\cP(a)$ does not contain $S_3$ or $S_1$, however, it contains $P_1\oplus S_1$ and $I_2\oplus S_1$ since $S_1$ is not a strict quotient of these two objects of $\cG$.
}

\begin{figure}[htbp]
\begin{center}
\begin{tikzpicture}

\begin{scope}
\coordinate (A) at (-1.25,-2);
\coordinate (B) at (1.95,-.5);
\draw (0,0) ellipse[x radius= 2cm,y radius=1.73cm];
\draw[fill,white] (-1,0) circle[radius=2cm];
\draw[very thick] (-1,0) circle[radius=2cm];
\draw (1,0) circle[radius=2cm];
\draw[very thick] (0,-2) ellipse[x radius=3cm,y radius=2cm];
\draw[very thick] (A)..controls (1,-3.5) and (2,-2)..(B);
\coordinate (G0) at (-3,-3.5);
\draw (G0) node{$\cG_0$};
\coordinate (G1) at (-1.5,-3);
\draw (G1) node{$\cG_1$};

\coordinate (G2) at (-1.7,.8);
\draw (G2) node{$\cG_2$};
\coordinate (G3) at (-2,-1);
\draw (G3) node{$\cG_3$};
\coordinate (G6) at (-.9,-1.5);
\draw (G6) node{$\cG_6$};
\coordinate (G9) at (0,-.6);
\draw (G9) node{$mod\text-\Lambda$};
\coordinate (G4) at (.5,-2.9);
\draw (G4) node{$\cG_4$};
\coordinate (G5) at (0.2,-2.2);
\draw (G5) node{$\cG_5$};

\coordinate (D1) at (2.8,-3.5);
\draw (D1) node{$D(S_1)$};
\coordinate (D2) at (-3,1.5);
\draw (D2) node{$D(S_3)$};

\coordinate (D3) at (0,-3.5);
\draw (D3) node{$D(I_2)$};

\coordinate (D4) at (1.4,-2.5);
\draw (D4) node{$D(P_1)$};

\clip (0,-2) ellipse[x radius=3cm,y radius=2cm];
\draw[very thick] (.65,-1.3) circle[radius=2.04cm];

\end{scope}

\begin{scope}[xshift=8cm];
\begin{scope}
\draw[very thick,gray!60!white,fill] (.65,-.8) ellipse[x radius=2.04cm, y radius=2.7cm];
\draw[very thick,white,fill] (0,-2) ellipse[x radius=3cm,y radius=2cm];
\end{scope}
\draw[blue](1.2,0) node{$a$};
\draw[red] (2,-.6)node[left]{$b$};
\draw[green!80!black] (0,0)node{$\bullet$}node[below]{$c$};

{
\coordinate (A) at (-1.2,-2);
\coordinate (B) at (1.95,-.5);
\coordinate (C) at (0,1.73);
\draw[very thick] (-1,0) circle[radius=2cm];
\draw[very thick] (0,-2) ellipse[x radius=3cm,y radius=2cm];
\draw[very thick] (A)..controls (1,-3.5) and (2,-2)..(B)..controls (2,0.7) and (1,1.7)..(C);
\coordinate (G0) at (-3,-3.5);
\draw (G0) node{$\cP_0$};
\coordinate (G1) at (-1.5,-3);
\draw (G1) node{$\cP_1$};

\coordinate (G2) at (-1.7,.9);
\draw (G2) node{$\cP_2$};
\coordinate (G3) at (-2,-1);
\draw (G3) node{$\cP_3$};

\coordinate (G6) at (-.3,-.7);
\draw (G6) node{$\cP_6=\cG_6$};
\draw (-.3,-1.3) node{$=\cG$};

\draw (0,.9) node{$\cP_7$};
\draw (1.1,.9) node{$\cP_8$};
\draw (1.9,.9) node{$\cP_9$};

\coordinate (G9) at (0,-.6);
\coordinate (G4) at (.5,-2.9);
\draw (G4) node{$\cP_4$};
\coordinate (G5) at (0.2,-2.2);
\draw (G5) node{$\cP_5$};

\coordinate (D1) at (2.8,-3.5);
\draw (D1) node{$\tiny D_\cG(S_1)$};
\coordinate (D2) at (-3,1.5);
\draw (D2) node{$\tiny D_\cG(S_3)$};

\coordinate (D3) at (0,-3.5);
\draw (D3) node{$\tiny D_\cG(I_2)$};

\coordinate (D4) at (1.2,-1.7);
\draw (D4) node{$\tiny D_\cG(P_1)$};

\draw[very thick] (.65,-.8) ellipse[x radius=2.04cm, y radius=2.7cm];
}
\draw[thick, green!80!black,<-] (-1,-1.5)--(-3.5,-1.5) node[left]{$\gamma$};

\end{scope}

\end{tikzpicture}
\caption{The full picture for $mod\text-\Lambda$ is on the left. The torsion classes $\cG_i$ contained in $\cG=\cG_6$ for $i\le 6$ are equal to the pseudo-torsion classes $\cP_i$ for $i\le 6$ on the right. These are the pseudo-torsion classes for the six open pseudo-chambers on the right. There are three more pseudo-chambers in the reduced diagram. These are shaded since they are cut up into infinitely many pseudo-chambers by the pseudo-walls $D_\cG(aP_1\oplus bI_2\oplus cS_1)$ where $aP_1\oplus bI_2\oplus cS_1$ are pseudo-bricks when $c>0$. Vertices $\color{blue}a$, $\color{red}b$ and $\color{green!80!black}c$ are used to compute $\cP_8,\cP_9,\cP_7$. \ki{Green paths, such as $\gamma$, determine generators and cogenerators for the pseudo-torsion classes. (See Remark \ref{rem: gamma gives generators and cogenerators}.)}
}
\label{Fig: A3 example}
\end{center}
\end{figure}
}

\ki{\begin{rem}\label{rem: gamma gives generators and cogenerators}
The open regions have pseudo-torsion classes which, in this example, are functorially finite (Definition \ref{def: ff p-tor classes} below) with generators and cogenerators determined by the walls being crossed by green paths going through the region. In Figure \ref{Fig: A3 example} we can see that $\cP_3$ is generated by $S_1$ and $S_3$ since the green path $\gamma$ passes through the walls $D_\cG(S_1),D_\cG(S_3)$ before reaching the pseudo-chamber and $\cP_3$ is cogenerated by $I_2$ since $\gamma$ crosses $D_\cG(I_2)$ before reaching the ``positive pseudo-chamber'' labeled $\cG$. This observation is Theorem \ref{thm: green paths gives generators and cogenerators} below.
\end{rem}
}

\subsubsection{Ghosts in $A_3$}\label{eg: A3 ghosts}
 The next example is the stability diagram for Example \ref{eg: A3 example}. This example is slightly different from Example \ref{eg: A3 example stability}. We will see ``ghosts''. This example comes from Definition 2.3 and Figure 4 in \cite{GrInvRedrawn} where the study of ghosts originated. We use the ``balanced'' notation for $\tau$-rigid pairs explained in \cite{IT26}.

 In Figure \ref{Figure03} we use the ``balanced'' notation $\symm AB$ for $\tau$-rigid pairs explained in \cite{IT26}. Although this notation is not defined for torsion classes since we do not have the Auslander--Reiten translation $\tau$ on $\cG$, we can use the Theorem in \cite{IT26} which predicts what the balanced notation should be from the stability diagram. It is $\symm AB$ if $D_\cG(A)$ is the only wall that ``covers'' the vertex and $A=0$ if there is no covering wall and $D_\cG(B)$ is the only wall ``underneath'' the vertex, with $B=0$ if none exist. 

For example we have the vertex on the far left labeled $\symm0{S_3}$ indicating that there is no covering wall and the only wall underneath is $D_\cG(S_3)$. On the right we have four interesting vertices:
\begin{enumerate}
\item At top we have a vertex with label $\symm{I_2}{P_2}$. This suggests that $\tau I_2=P_2$ which seems correct.
\item Lower we have a vertex labeled $\symm{P_3}{S_2}$ which suggests that $\tau P_3$ should be $S_2$.
\item We have a vertex with label $\symm{S_2}X$ where, according to the convention we should have $X=P_2\oplus S_2$.
\item Finally we have label $\symm{Y}0$ where $Y=P_2\oplus S_2$.
\end{enumerate}

The missing object in our torsion class is $S_1$ which is a simple projective. Although $S_1$ is not in $\cG$, $\nu S_1=P_3$ is in $\cG$. So, this appears at the top black spot in Figure \ref{Figure03}. This is the $g$-vector of $S_1[1]$ which, in balanced notation, is $\symm{0}{\nu S_1}=\symm 0{P_3}$. The missing object $S_1$ produces two ghosts. (See \cite{GrInvRedrawn}, \cite{MoreGhosts} for more details.) The two ghosts are the non-strict epimorphisms
\[
	P_2\to S_2 \qquad \text{and}\qquad P_3\to I_2.
\]
These appear at vertices (3) and (4) as $X=P_2\oplus S_2$ and $Y=P_3\oplus I_2$. Also, at the bottom black vertex we have $\symm{P_2}0$ instead of $\symm{S_1}0$.

Something similar happens in Figure \ref{Fig: A3 example}. The vertex labeled $\color{blue}a$ is the $g$-vector of the missing projective object $P_2$. The balanced notation is $\symm{P_1}0$ which matches the unique ghost of $P_2$ which is the epimorphism $P_1\to S_1$. The other missing object in that example is $S_2$ with $\tau S_2=S_3\in\cG$. The balanced notation is $\symm{I_2}{S_3}$ which matched the unique ghost of $S_2$ which is $I_2\to S_1$.

Finally, we would like to point out that the balanced notation gives the pseudo-wide subcategory in each case. The pseudo-wide subcategory containing the vertex with balanced $g$-vector given by $\symm AB$ is $A^\perp\cap\,^\perp B$ which agrees with the construction in \cite{IT26}.
{
\begin{figure}[htbp]
\begin{center}

\begin{tikzpicture}[scale=1.5] 

\begin{scope}[xshift=3.3mm,yshift=-6.9mm]
\begin{scope}[rotate=63]
\clip rectangle (2,1.3) rectangle (0,-1.3);
\draw[fill,gray!50!white] (0,0) ellipse [x radius=1.4cm,y radius=1.18cm]; 
\end{scope}
\end{scope}

\draw[fill,white] (0,0) ellipse [x radius=1.4cm,y radius=1.2cm]; 

\begin{scope}[yshift=-1.35cm]
	\draw[fill,gray!50!white] (0,0) circle[radius=1.4cm]; 
\end{scope}
\begin{scope}[xshift=-.7cm]
    	\draw[fill,white] (0,0) circle[radius=1.4cm]; 
\end{scope}


\draw[blue] (0,-1.45) node{\small$P_2$};
\draw (0,-1.22) node{$\bullet$};
\draw[red] (0,-1) node{\small$0$};
\draw[blue] (0,1.4) node{$0$};
\draw (0,1.2) node{$\bullet$};
\draw[red] (0,1) node{\small$P_3$};
\draw[blue] (-1.5,-1.3) node{$0$};
\draw[red] (-1.2,-1.1) node{\small$S_3$};
\draw[blue] (1.5,-1.3) node{$0$};
\draw[red] (1.2,-1.1) node{\small$I_2$};
\draw[blue] (-.9,0) node{\small $S_2$};
\draw[red] (-.55,-.25) node{\small$0$};
\draw[blue] (.85,-.05) node{\small $Y$};
\draw[red] (.55,-.25) node{\small$0$};
\draw[blue] (.6,.8) node{\small $I_2$};
\draw[red] (.4,.45) node{\small$P_2$};
\draw[blue] (1.4,-.6) node{\small $P_3$};
\draw[red] (1,-.6) node{\small$S_2$};
\draw[blue] (1.55,.2) node{\small $S_3$};
\draw[red] (1.23,.05) node{\small$X$};
{
\begin{scope}

\begin{scope}
\clip rectangle (2,1.3) rectangle (0,-1.3);
\draw[thick] (0,0) ellipse [x radius=1.4cm,y radius=1.2cm]; 
\end{scope}

\begin{scope}[xshift=3.3mm,yshift=-6.9mm]
\begin{scope}[rotate=63]
\clip rectangle (2,1.3) rectangle (0,-1.3);
\draw[thick] (0,0) ellipse [x radius=1.4cm,y radius=1.18cm]; 
\end{scope}
\end{scope}

\begin{scope}[xshift=.87mm]
		\draw[thick] (1.4,.9) node[left]{\tiny$D(I_2)$}; 
		\draw[thick] (1.66,-.4) node{\tiny$D(P_3)$}; 
\end{scope}
\begin{scope}[xshift=-.7cm]
	\draw[thick] (0,0) circle[radius=1.4cm]; 
		\draw (-1,1.1) node[left]{\small$D(S_2)$}; 
\end{scope}
\begin{scope}[xshift=.7cm]
	\draw[thick] (0,0) circle[radius=1.4cm]; 
		\draw (1,1.1) node[right]{\small$D(S_3)$};
\end{scope}
\begin{scope}[yshift=-1.35cm]
	\draw[thick] (0,0) circle[radius=1.4cm]; 
	\draw (1.15,-.8) node[right]{\small$D(P_2)$}; 
\end{scope}
%
\coordinate(B1) at (-1.39,-1.22);
\coordinate(B2) at (.69,-.14);
\coordinate(B) at (-.9,-.95);
\coordinate(A) at (1.4,.16);
\coordinate(C) at (-.48,-.72);
\coordinate(AC) at (-0.05,-.4);
\end{scope} 
}
\end{tikzpicture}

\caption{This is the \ki{reduced stability diagram} for Example \ref{eg: A3 ghosts}. The two back spots are the $g$-vectors of $S_1$ at bottom and $S[1]$ at top. In the quadrilateral we have vertices with balance $g$-vector notation $\symm{S_3}{X}$ and $\symm{Y}{0}$ explained in the text. For lack of space, we have shortened the pseudo-wall notation in this figure from $D_\cG(X)$ to $D(X)$. \ki{There are 9 open pseudo-chambers and the 4 shaded regions which are unions of infinitely many pseudo-chambers. Generators and cogenerators for the pseudo-torsion classes for the nine open pseudo-chambers can be determined using green paths as in Remark \ref{rem: gamma gives generators and cogenerators}.}
}
\label{Figure03}
\end{center}
\end{figure}
}

\subsection{Walls have two sides}\label{ss23: p-tor for walls}

\begin{thm}\label{thm: P into overline P onto W}
$\overline \cP(\theta)$ is the pseudo-torsion class generated by $\cP(\theta)$ and $\cW(\theta)$. More precisely, every object $M$ of $\overline P(\theta)$ has a unique strict subobject $M_0$ which lies in $\cP(\theta)$ so that the quotient $M/M_0$ is an object of the pseudo-wide subcategory $\cW(\theta)$. Equivalently,
\[
	\cP(\theta)^\perp\cap \overline \cP(\theta)=\cW(\theta).
\]
\end{thm}

{
\begin{proof}
Let $M$ be an object of $\overline\cP(\theta)$ which is not in $\cP(\theta)$. Then, $\theta(M')\ge0$ for all strict quotients $M'$ of $M$. If $M\notin\cP(\theta)$ then equality must hold for some nonzero strict quotient $M'$ of $M$. Take $M'$ maximal with this property. Then $\theta(M')=0$ and $\theta(M'')\ge0$ for all strict quotients $M''$ of $M'$. So, $M'\in \cW(\theta)$. If $M'=M$ we are done. Otherwise $\theta(M)>0$ by maximality of $M'$. Let $M_0\cof M$  be the kernel of $M\onto M'$. Then $\theta(M_0)=\theta(M)>0$, for any proper strict subobject $N'\cof M_0$ we have that $M/N'$ is a strict quotient of $M$ larger than $M'$. So, $\theta(M/N')=\theta(M_0/N')>0$. So, $M_0\in \cP(\theta)$ giving the desired strict filtration $M_0\cof M\onto M'$ with $M_0\in\cP(\theta)$ and $M'\in\cW(\theta)$.
\end{proof}
}

\begin{cor}\label{cor: for thin walls, ov P is min ext of P}
For $\theta$ on a thin wall $D_\cG(X)$, $\overline\cP(\theta)$ is a minimal extension of $\cP(\theta)$ with extending object $X$.
\end{cor}

\begin{proof}
Let $\cP$ be a pseudo-torsion class in $\cG$ so that $\cP(\theta)\subseteq \cP\subseteq \overline\cP(\theta)$. Then either $X\in \cP$ or it isn't. If $X\in \cP$ then $\cP=\overline\cP(\theta)$ since this is generated by $\cP(\theta)$ and $X$. If $X\notin \cP$ then every object $P\in\cP$ must lie in $\cP(\theta)$ since $P$ has a filtration $P_0\into P\onto W$ where $W\in \cW(\theta)$. Since this is a thin wall, $W$ must be an iterated self-extension of $X$. So, $W$, being a strict quotient of $P$, must be 0.
\end{proof}

Every pseudo-wall has a negative or ``green'' side and a positive or ``red'' side. We also need to consider ``virtual green walls'' and ``virtual red walls". \ki{The idea comes from Dedekind cuts. These are partitions of $\QQ$, the set of rational numbers into two nonempty parts $\QQ=P\coprod Q$ where every $p\in P$ is less than every $q\in Q$. Every irrational number gives one Dedekind cut, but every rational number gives two Dedekind cuts, with the rational number assigned to $P$ or $Q$. If we think of rational numbers as walls and irrational numbers as pseudo-chambers, we see that there are two more ``virtual chambers'' at each rational number one on each side of the rational number. This is the idea of virtual chambers on two sides of each wall.} 

{\color{black}
Every pseudo-wall $D_\cG(X)$ has, at each point $\theta$, a negative or ``green'' side and a positive or ``red'' side. On each side, say the green side, there are two cases:
\begin{enumerate}
    \item There is an open pseudo-chamber $\cU$ on the green side of $\theta$. This means that $\theta-\varepsilon\eta\in\cU$ for sufficiently small $\varepsilon>0$.
    \item $\theta\in D_\cG(X)$ is in the limit of a sequence of walls on its green side. This means there is a sequence of pseudo-walls $D_\cG(X_k)$ and points $\theta-a_k\eta\in D_\cG(X_k)$ where $a_k$ are positive real numbers converging to $0$.
\end{enumerate}
In the second case we will construct a ``virtual green chamber'' which is on the green side of the wall. The idea is that these are the points on the negative side of $D_\cG(X)$ at $\theta$ but on the positive side of all $D_\cG(X_k)$. A moment's thought will tell us that there are no points in $V_\Lambda\cong \RR^n$ satisfying this condition. A more accurate description is that it is a set of cotangent vectors at $\theta$ which are normal to $D_\cG(X)$. But we prefer a more elementary description.
}

{\color{black}
\begin{defn}\label{def: virtual chambers}
For any $\theta_0\in \cW(\cG)$,  the \textbf{virtual green chamber} associated to $\theta_0$ is the set of all $\theta\in \cW(\cG)$ so that $\cP(\theta)=\cP(\theta_0)$. Similarly, the \textbf{virtual red chamber} associated to $\theta_0$ is the set of all $\theta\in \cW(\cG)$ so that $\overline\cP(\theta)=\overline\cP(\theta_0)$. We will speak about these as separate sets even though they might be the same pointset. For example, $\theta_0$ always lies in exactly two virtual chambers, the red side and the green side.
\end{defn}
}

Both sides, together with the virtual walls will be collected into disjoint unions of path connected subsets which we call ``green components'' and ``red components''. Thus the set $\cW_-(\cG)$ will have a ``green partition'' and $\cW_+(\cG)$ will have a ``red partition'' which we can now construct.

\begin{defn}\label{def: side components of walls}
    Let $\theta_0\in\cW_-(\cG)$. Then we define the \textbf{green component} of $\theta_0$ to be the set of all $\theta\in\cW_-(\cG)$ so that $\cP(\theta)=\cP(\theta_0)$. The \textbf{red component} of any $\theta_0\in \cW_+(\cG)$ is defined similarly: It is the set of all $\theta\in\cW_+(\cG)$ so that $\cQ(\theta)=\cQ(\theta_0)$. (Equivalently, by Theorem \ref{thm: duality between P(theta)bar and Q(theta)}, $\overline\cP(\theta)=\overline\cP(\theta_0)$.) \ki{In the special case when there are no pseudo-chambers $\cU$ with $\cP(\cU)=\cP(\theta_0)$, we call the green component an \textbf{infinitesimal green chamber}.} \textbf{Infinitesimal red chambers} are defined analogously. They are the red components of some $\theta_0\in\cW_+(\cG)$ which contain all $\theta\in V_\Lambda$ so that $\overline\cP(\theta)=\overline\cP(\theta_0)$.
\end{defn}

\ki{It follows from Corollary \ref{cor: P,Q equal on convex sets} that infinitesimal green chambers are convex and similarly for infinitesimal red chambers. In the case when there is a pseudo-chamber $\cU$ having pseudo-torsion class $\cP(\cU)$ equal to $\cP(\theta_0)$ for some $\theta_0\in\cW(\cG)$, the closure of $\cU$ will contain $\cU$, the green components of $\theta_0$ and the red components of all $\theta_1$ so that $\overline\cP(\theta_1)=\cP(\cU)$. These are the ``red'' and ``green'' boundary points of $\cU$.}


\section{Harder-Narasimhan stratification}\label{new sec3}

\subsection{Green paths and FHO sequences}\label{ss31: FHO}
{%
We show that ``green paths'' satisfy a ``strict forward Hom-orthogonality'' condition. This will lead to a ``strict Harder-Narasimhan stratification'' of our torsion class $\cG$.

\begin{defn}\label{def: smooth Bridgeland stability}
By a \textbf{smooth green path} in $V_\Lambda$ we mean a smooth path $\theta_t$, $t\in \mathbb{R}$, satisfying the following conditions.
\begin{enumerate}
\item $\frac{\partial}{\partial t}\theta_t(M)>0$ whenever $\theta_t\in D_\cG(M)$, i.e., whenever $\theta_t(M)=0$ and $\theta_t( M')\le0$ for all strict subobjects $M'$ of $M$.
\item For sufficiently large $t$, we have $\theta_t(M)>0$ for all $M\in\cG$.
\item For sufficiently large $t$, we have $\theta_{-t}(M)<0$ for all $M\in\cG$.
\end{enumerate}
A \textbf{linear green path} is a path of the form $\theta_t=\theta_0+t\eta$ for $\eta$ any real vector with all coordinates being positive.
 \end{defn}
}

{%
\begin{defn}\label{def: W(theta)}
    For any stability condition $\theta$, we recall that $\cW(\theta)$ is the set of all objects $X$ in $\cG$ so that $\theta\in D_\cG(X)$. Let $\cW_0(\theta)$ be the set of all $X$ in $\cW(\theta)$ so that no proper strict subobject of $X$ is in $\cW(\theta)$. Equivalently, $\theta$ is in the interior of $D_\cG(X)$.
\end{defn}

Note that every object $X\in\cW_0(\theta)$ is a pseudo-brick. (Every nonzero strict endomorphism of $X$ is an automorphism of $X$.)
}

{%
\begin{lem}\label{lem: green is like linear}
Suppose $\theta_t$ is a green path and \ki{$X\in \overline\cP(\theta_{t_0})$}. Then $\theta_t(X')>0$ for all $t>t_0$ and all nonzero strict quotients $X'$ of $X$ including $X'=X$.
\end{lem}

\begin{proof}
By Definition \ref{def: smooth Bridgeland stability} (1), $\theta_t(X')>0$ for $t=t_0+\varepsilon$ for small $\varepsilon>0$ for all nonzero strict quotients $X'$ of $X$, including $X'=X$. 

To show that $\theta_t(X')>0$ for all $t>t_0$ and all nonzero strict quotients $X'$ of $X$, suppose not. Then, we have $\theta_t(X')\le 0$ for some $t>t_0$ and some $X'$. Since the set of all dimension vectors of such $X'$ is finite, the set of such $t$ is a closed subset of the interval $(t_0,1]$ which does not have $t_0$ as a limit point. So, there is a minimum such $t$, call it $t_1>t_0$. Then $\theta_{t_1}(X')=0$ and $\theta_{t_1}(X'')\ge0$ for all other strict quotients $X''$ of $X$. Therefore, $\theta_{t_1}\in D_\cG(X')$. By Definition \ref{def: smooth Bridgeland stability} (1), $\theta_t(X')$ is strictly increasing at $t=t_1$. So, $\theta_t(X')<0$ for $t=t_1-\varepsilon$ contradicting the minimality of $t_1$. This proves the lemma.
\end{proof}
}
\ki{Lemma \ref{lem: green is like linear} can be rephrased as follows.

\begin{thm}\label{thm: P(theta) is increasing}
Let $\theta_t$ be a green path and let $t_0<t_1$. Then
\[
	\overline\cP(\theta_{t_0})\subseteq \cP(\theta_{t_1}).
\]
\end{thm}
}
{
\begin{thm}\label{thm: FHO}
    Let $\theta_t$ be a smooth green path and let $t_0<t_1$. Then, there are no nonzero strict morphisms $X_0\to X_1$ where $X_i\in \cW(\theta_{t_i})$. Similarly, there are no nonzero strict morphisms between nonisomorphic objects of $\cW_0(\theta_t)$.
\end{thm}

\begin{proof}
    Suppose $f:X_0\to X_1$ is a nonzero strict morphism. Then, its image $I$ is a nonzero strict quotient of $X_0$ and a strict subobject of $X_1$. Since $\theta_{t_0}\in D(X_0)$, $\theta_{t_0}(X_0)=0$ and $\theta_{t_0}(I)\ge0$. By the lemma, $\theta_{t_1}(I)>0$. But, $\theta_{t_1}(I)\le0$ since $I$ is a strict subobject of $X_1$. This is a contradiction.

    Now suppose there is a nonzero strict morphism $f:X_0\to X_1$ where $X_0,X_1\in \cW_0(\theta_t)$ are nonisomorphic. Then, $\ker f,\im f,\coker f$ lie in $\cW(\theta_t)$ by Theorem \ref{thm: W(theta) is W}. By definition of $\cW_0(\theta_t)$, we must have $X_0=\im f=X_1$ contradicting the assumption that $X_0\not\cong X_1$.
\end{proof}
}

{
Using the following definition we can rephrase Theorem \ref{thm: FHO}.

\begin{defn}\label{def: FHO}
By a \textbf{strict forward Hom-orthogonal (FHO) sequence} of pseudo-bricks we mean a set $\cB$ of pseudo-bricks together with an indexing function $\pi:\cB\to\RR$ so that there are no nonzero strict morphisms $B\to B'$ (except for isomorphisms) if $\pi(B)\le \pi(B')$.
\end{defn}

Theorem \ref{thm: FHO} can be rephrased by saying that a smooth green path $\theta_t$ gives a strict FHO sequence, namely, the collection of all pseudo-bricks in $\cW_0(\theta_t)$ for all $t$ with indexing function $\pi(B)=t$ if $B\in\cW_0(\theta_t)$. In the next subsection we will show that this strict FHO sequence is maximal. 
}
\subsection{HN-stratifications}\label{ss32: HN}

{
We show that each smooth green path $\theta_t$ for a torsion class $\cG$ gives a ``strict HN-stratification'' of $\cG$ that is minimal and conclude that the corresponding ``strict Hom-orthogonal sequence'' is maximal.

\begin{defn}\label{def: strict filtration}
By a \textbf{strict filtration} of an object $X\in \cG$ we mean a sequence of cofibrations:
\[
	X_0\cof X_1\cof \cdots\cof X_n=X.
\]
By Theorem \ref{def/thm: subquotient}, this is equivalent to an increasing sequence $X_1\subset X_2\subset\cdots\subset X_n=X$ where each $X_i$ is a strict subobject of $X$. Then $X_i\cof X_j$ is strict for all $i<j$. The subquotients $X_j/X_{j-1}$ are {strict subquotients} of $X$ by Theorem \ref{def/thm: subquotient}.
\end{defn}

\begin{lem}\label{lem: filtration of W(theta)}
Any object $X\in \cW(\theta)$ has a strict filtration having strict subquotients in $\cW_0(\theta)$.
\end{lem}

\begin{proof}
If $X$ is not in $\cW_0(\theta)$, it has a subobject $K$ so that $\theta(X/K)=0$. Then $\theta(K)=0$. So $K,X/K\in \cW(\theta)$ and, by induction on length, $K,X/K$ have filtrations with subquotients in $\cW_0(\theta)$. Pulling back the filtration of $X/K$ to $X$, we obtain the desired filtration of $X$.
\end{proof}
}

{
\begin{thm}\label{thm: HN stratification}
Let $\theta_t$ be a smooth green path. Then, for any object $M\in\cG$, there is a sequence of real numbers $t_1>t_2>\cdots>t_m$ and a filtration of $M$ by strict subobjects
\[
	M=M_0\supset M_1\supset M_2\supset \cdots\supset M_m=0
\]
in $\cG$ so that $M_{k-1}/M_k\in\cW(\theta_{t_k})$. Furthermore, this filtration can be refined: Each $M_k$ has a filtration by {\color{black}strict subobjects:}
\[
	M_{k0}=M_{k+1}\subset M_{k1}\subset\cdots\subset M_{kn}=M_k
\]
so that the quotients of consecutive terms lie in $\cW_0(\theta_{k+1})$.
\end{thm}

\begin{proof}
First note that there are only finitely many real numbers $t$ having the property that $\theta_t(X)=0$ for some strict quotient $X$ of $M$. This is because the condition $\theta_t(X)=0$ depends only on the dimension vector of $X$ and, although $M$ may have infinitely many strict quotients, there are only finitely many dimension vectors of strict quotients. In this finite set, let $t_1$ be maximal so that $\theta_{t_1}(X)=0$ for some nonzero strict quotient $X$ of $M$. Such $t_1$ exists since $\theta_t(X)>0$ for $t>>0$ and $\theta_t(X)<0$ for $t<<0$. 

Then, $\theta_{t_1}(X')\ge 0$ for all other strict quotients $X'$ of $M$ since, if $\theta_{t_1}(X')<0$, then $\theta_t(X')=0$ for some $t>t_1$ contradicting the maximality of $t_1$.

Let $X_1$ be a maximal strict quotient of $M$ such $\theta_{t_1}(X_1)=0$ and let $M_1$ be the kernel of the epimorphism $M\onto X_1$. Then $\theta_{t_1}(X')\ge0$ for all strict quotients of $X_1$ since every strict quotient of $X_1$ is a strict quotient of $M$. So, $X_1=M/M_1\in \cW(\theta_{t_1})$. If $M_1=0$, we have the desired filtration of $M$ with only one term. So, we are done.

If $M_1\neq0$ then, by maximality of $X_1$ we have $\theta_{t_1}(M_1)=\theta_{t_1}(M)>0$. Also, we have $\theta_{t_1}(Y)>0$ for all nonzero strict quotients of $M_1$ since $M$ modulo the kernel of $M_1\onto Y$ is a strict quotient of $M$ larger than $X_1$. Thus, $M_1\in \cP(\theta_{t_1})$. 

Let $t_2<t_1$ be maximal so that $\theta_{t_2}(Y)=0$ for some nonzero strict quotient $Y$ of $M_1$. Let $X_2$ be a maximal strict quotient of $M_1$ so that $\theta_{t_2}(X_2)=0$ and let $M_2$ be the kernel of $M_1\to X_2$. Then $X_2=M_1/M_2\in \cW(\theta_{t_1})$ and $M_2\in \cP(\theta_{t_2})$. Continuing in this way, we get the first filtration of $M$.

The second filtration follows from Lemma \ref{lem: filtration of W(theta)}.
\end{proof}
}

{
This implies that we have a ``strict HN-stratification'' of $\cG$ in the following sense.

\begin{defn}\label{def: HN-stratification}
By a \textbf{strict HN-stratification} of $\cG$ we mean a family of pseudo-wide subcategories $\cW_s,s\in \cS$ of $\cG$ parametrized by a totally ordered set $\cS$ satisfying the following.
\begin{enumerate}
\item There are no nonzero strict morphisms from $X\in \cW_{s_0}$ to $Y\in \cW_{s_1}$ if $s_0<s_1$.
\item Every object $M$ of $\cG$ has a strict filtration
\[
	M=M_0\supset M_1\supset M_2\supset \cdots \supset M_{m-1}\supset M_{m}=0
\]
so that $M_{k-1}/M_k$ lies in $\cW_{s_k}$ where $s_1>s_2>\cdots >s_m$. We call this the \textbf{strict HN-filtration} of $M$ with respect to the strict HN-stratification $\cW_s,s\in \cS$.
\end{enumerate}
The pseudo-wide subcategories $\cW_s$ are called the \textbf{slices} of the HN-stratification.
\end{defn}
}

{
\begin{lem}\label{lem: induced maps are strict}
Let $f:M\to N$ be a strict morphism and let $X$ be a subobject of the kernel of $f$. Then the induced map $\overline f:M/X\to N$ is strict.
\end{lem}

\begin{proof}
Since $f,\overline f$ have the same image, we may assume they are both onto. Then $f:M\onto N$ is a strict epimorphism with kernel, say $K$ and $X\subseteq K$. By Lemma \ref{lemma A}, $K/X\cof M/X$ is a cofibration. So, the quotient map $M/X\onto (M/X)/(K/X)\cong M/K\cong N$ is strict.
\end{proof}

\begin{lem}\label{lem: no forward Hom}
{\color{black} Let $M$ have strict HN-filtration as given above. Then there is no nonzero strict morphism from $M$ to any $Y\in\cW_s$ where $s>s_1$.}
\end{lem}

\begin{proof}
Any strict morphism $f: M\to Y$ must be 0 on $M_{k-1}$ since $s_k<s_1<s$. The induced map $\overline f:M/M_{k-1}\to Y$ is strict by {Lemma \ref{lem: induced maps are strict}}. By induction on $k$, $\overline f=0$.
\end{proof}
}

{
\begin{prop}\label{prop: uniqueness of HN}
The strict HN-filtration of $M$ with respect to a fixed strict HN-stratification of $\cG$ is unique.
\end{prop}

\begin{proof} Given a strict HN-filtration of $M$ as in Definition \ref{def: HN-stratification},
let $N=M$ with another strict HN-filtration
\[
	N=N_0\supseteq N_1\supseteq N_2\supseteq \cdots \supseteq N_{n-1}\supseteq N_{n}=0
\]
with $N_{j-1}/N_j\in \cW_{t_j}$. We claim that $t_1=s_1$. Otherwise, we have an inequality, say $t_1>s_1$. By Lemma \ref{lem: no forward Hom} this would imply that there are no strict morphisms $M\to N/N_1\in\cW_{t_1}$. In other words, $M\subset N_1$ a contradiction since $M=N$.

So, $t_1=s_1$. Then, the same argument shows that $N_1\subseteq M_1$ and $M_1\subseteq N_1$. So, $N_1=M_1$. By induction on $n,m$ we get $N_k=M_k$ for all $k$ and $n=m$. {So, these two filtrations of $M$ are the same and the filtration is unique.}
\end{proof}
}

{
\begin{thm}\label{thm: HN for P}
    Given a smooth green path $\theta_{t}$ and $t_0\in\RR$, an object $M\in\cG$ lies in $\cP(\theta_{t_0})$ if and only if the strict HN-filtration of $M$ has strict subquotients in $\cW(\theta_t)$ for $t<t_0$.
\end{thm}

\begin{proof}
    This follows from the proof of Theorem \ref{thm: HN stratification}. Given $M\in \cP(\theta_{t_0})$, we have that $\theta_{t_0}(X)>0$ for any nonzero strict quotient $X$ of $M$. So, $\theta_{t_1}(X)=0$ implies that $t_1<t_0$. The other $t_i$ are $<t_1$. So, all strict subquotients of $M$ lie in some $\cW(\theta_{t_i})$ where $t_i\le t_1<t_0$.

    Conversely, suppose $M\notin\cP(\theta_{t_0})$. Then $\theta_{t_0}(X)\le0$ for some nonzero strict quotient $X_1$ of $M$. So, there is some $t_1\ge t_0$ so that $\theta_{t_1}(X_1)=0$ for some strict quotient $X_1$ of $M$. As in the proof of Theorem \ref{thm: HN stratification}, we take $X_1$ maximal and obtain $X_1\in \cW(\theta_{t_1})$. This proves the theorem.
\end{proof}

To obtain the analogous theorem for $\cQ(\theta_{t_0})$ we need to build the strict HN-filtration of any $M\in\cG$ from the bottom up.

\begin{lem}\label{lem: bottom of M}
    For any $M\in\cG$, let $t_1$ be minimal so that $M$ has a strict subobject $Y$ with $\theta_{t_1}(Y)=0$. The maximal such $Y$ is the bottom layer of the unique strict HN-filtration of $M$, i.e., $M/Y$ has HN-filtration with strict subquotients in $\cW(\theta_t)$ for $t>t_1$.
\end{lem}

\begin{proof}
    Minimality of $t_1$ implies that $Y\in\cW(\theta_{t_1})$. Maximality of $Y$ implies that the other subquotients of $M$ must lie in $\cW(\theta_t)$ for $t>t_1$.
\end{proof}

\begin{thm}\label{thm: HN for Q}
       Let $\theta_{t_0}$ be a point on a smooth green path. Then an object $M\in\cG$ lies in $\cQ(\theta_{t_0})$ if and only if the strict HN-filtration of $M$ has strict subquotients in $\cW(\theta_t)$ for $t>t_0$.
\end{thm}

{
\begin{proof}
    Let $M\in \cQ(\theta_{t_0})$. Then $\theta_{t_0}(Y)<0$ for all nonzero strict subobjects $Y$ of $M$. So, $\theta_{t_1}(Y)=0$ implies $t_1>t_0$. By Lemma \ref{lem: bottom of M}, the bottom of the strict HN-filtration of $M$ lies in $\cW(\theta_{t_1})$ for $t_1>t_0$. The other strict subquotients lie in higher slices $\cW(\theta_t)$ for $t>t_1$. Conversely, suppose that the bottom subquotient $Y$ of $M$ lies in $\cW(\theta_{t_1})$ where $t_1>t_0$. By minimality of $t_1$, this implies that $\theta_{t_0}(Y)<0$ for all nonzero strict subquotients $Y$ of $M$. So, $M\in\cQ(\theta_{t_0})$. 
\end{proof}
}

\begin{cor}
    Let $\theta\in V_\Lambda$ not lie on any pseudo-wall. Then every object $M$ of $\cG$ has a unique strict filtration $tM\cof M\onto fM$ where $tM\in \cP(\theta)$ and $fM\in \cQ(\theta)$.
\end{cor}

\begin{proof} This is Theorem \ref{thm: tM and fM}. We are just observing that it also follows from Theorem \ref{thm: HN for Q} above.
\end{proof}
}

{
We use the strict HN-filtration of any object of $\cG$ to obtain a maximality result for Theorem \ref{thm: FHO}. Recall that, for any smooth green path $\theta_t$, the collection $\cB$ of all pseudo-bricks in all $\cW_0(\theta_t)$ with indices $\pi(B)=t$ for $B\in \cW_0(\theta_t)$ forms a strict FHO sequence.

\begin{thm}\label{thm: maximality of FHO}
    This FHO collection of pseudo-bricks is maximal in the sense that no new brick $B_0$ can be added to the set with index $\pi(B_0)=t_0$ for some $t_0\in\RR$ without destroying the FHO property.
\end{thm}

\begin{proof}
    Suppose we attempt to add a pseudo-brick $B_0$ with $\pi(B_0)=t_0$. Then we will show that $B_0$ is isomorphic to an element of $\cB$ with index $t_0$.
    
    By Theorem \ref{thm: HN stratification}, there is a strict filtration of $B_0$ with strict quotient $B_1$ with index $t_1$ and strict subobject $B_n$ with index $t_n$ where $t_n\le t_1$. The strict morphism $B_0\to B_1$ indicates that $t_1\le t_0$ and the morphism $B_n\to B_0$ indicates that $t_0\le t_n$. Then $t_0=t_1=t_n$ and all strict subquotients of $B_0$ are in $\cW_0(\theta_{t_0})$. Since $B_0$ is a pseudo-brick, there is only one layer and $B_0=B_1\in\cW_0(\theta_{t_0})$ as claimed.
\end{proof}
 
The same proof shows the following.
\begin{thm}\label{thm: minimality of HN}
    The collection of bricks $\cB$ described above is minimal for the HN-stratification of $\cG$ in the sense that, if any element $B_0\in\cB$ were to be deleted, the rest of $\cB$ would not suffice to give a strict HN-filtration of $B_0$.\qed
\end{thm}
}

{

\subsection{Functorially finite pseudo-torsion classes}\label{ss: ff ptor classes}

In the standard case of $\cG=mod\text-\Lambda$, there is a bijection between functorially finite torsion classes, $\tau$-tilting pairs and ``standard open chambers''. This idea has a good reformulation in terms of ``balanced pairs'' \cite{IT26}. Namely, the functorially finite torsion classes are $Filt A=\,^\perp B$ where $A/B$ is a ``complete balanced pair'' for $mod\text-\Lambda$.

\begin{defn}\label{def: ff p-tor classes}
A pseudo-torsion class $\cP$ will be \textbf{functorially finite} if both $\cP$ and $\cP^\perp$ are finitely generated.
\end{defn}

\begin{prop}\label{prop: ff P and Q}
 If $\cP$ is generated by a set of objects $\cX$ then $\cQ=\cX^\perp$. Similarly, if $\cQ=\cP^\perp$ is generated by a set of objects $\cY$ then $\cP=\,^\perp\cY$.
\end{prop}

\begin{proof} Since $\cX\subseteq \cP$, we have $\cQ\subseteq \cX^\perp$ which gives $\,^\perp(\cX^\perp)\subseteq \,^\perp\cQ=\cP$. But clearly, $\cX\subseteq \,^\perp(\cX^\perp)$. Since $\cX$ generates $\cP$ we get $\,^\perp(\cX^\perp)=\cP$. So, 
\[
	\cQ=(\,^\perp(\cX^\perp))^\perp=\cX^\perp
\]
by Theorem \ref{thm: duality between P,Q}. The dual statement is proved in a similar way.
\end{proof}

\begin{lem}\label{lem: path in U}
Let $\theta_t$ be a green path and let $\theta_0$ be an element of a pseudo-chamber $\cU$.
\begin{enumerate}
\item For $t<0$, $\theta_t\in \cU$ if and only if $\cP(\theta_t)=\cP(\theta_0)$.
\item For $s>0$, $\theta_s\in\cU$ if and only if $\overline\cP(\theta_s)=\cP(\theta_0)$.
\end{enumerate}
\end{lem}

\begin{proof}
For (1), $\cP(\theta_t)\subseteq \overline\cP(\theta_t)\subseteq \cP(\theta_0)$ by Theorem \ref{thm: P(theta) is increasing}. If $\cP(\theta_t)=\cP(\theta_0)$ then $\cP(\theta_t)=\overline\cP(\theta_t)$ which implies $\cW(\theta_t)=0$. So, $\theta_t$ lies in another pseudo-chamber $\cU'$. By Theorem \ref{thm: P(theta) are different in different path components} we conclude that $\cU'=\cU$. Conversely, if $\theta_t\in\cU$ then $\cP(\theta_t)=\cP(\theta_0)$ by Theorem \ref{thm: P(theta) is constant on path components}.

The proof of (2) is similar starting with $\cP(\theta_0)\subseteq \cP(\theta_s)\subseteq \overline\cP(\theta_s)$.
\end{proof}

Note that, for any vector $w$ all of whose coordinates are positive, the linear path $\theta_t=tw$ is a green path since, whenever $\theta_t\in D_\cG(X)$ we have: 
\[
	\frac d{dt} \theta_t \undim X=w\cdot\undim X>0.
\]

\begin{thm}\label{thm: ff P,Q implies U open}
Suppose that $\cW(\theta_0)=0$ and $\cP(\theta_0),\cQ(\theta_0)$ are finitely generated. Then the pseudo-chamber containing $\theta_0$ is an open set.
\end{thm}

\begin{proof}
Let $\cX$ be a finite set of generators for $\cP(\theta_0)$ and $\cY$ a finite set of generators for $\cQ(\theta_0)$. Let $\theta_t$ be a green path passing through $\theta_0$. For each $X$ in $\cX$, let $X=X_0\supseteq X_1\supseteq X_2\supseteq \cdots \supseteq X_m=0$ be the strict HN-filtration of $X=X_0$ given by the path $\theta_t$. By Theorem \ref{thm: HN for P}, each strict subquotient $X_{k-1}/X_k$ lies in $\cW(\theta_{t_k})$ where $0>t_1>t_2>\cdots>t_m$. For any $t_1<t<0$, we have $X\in \cP(\theta_t)$ by the same theorem. Doing this for every $X\in\cX$, by taking the maximum $t<0$ we get that $\cX\subset \cP(\theta_t)$ which implies $\cP(\theta_0)\subseteq\cP(\theta_t)$ since $\cX$ generates $\cP(\theta_0)$. But $\cP(\theta_t)\subseteq \cP(\theta_0)$ by Theorem \ref{thm: P(theta) is increasing}. So, $\cP(\theta_t)= \cP(\theta_0)$.

Now take $Y\in\cY$. Since $Y\in\overline\cQ(\theta_0)$, we have, by Theorem \ref{thm: HN for Q}, that the strict HN-filtration of $Y$ has strict subquotients in $\cW(\theta_{s_k})$ where $0\le s_1<s_2<\cdots<s_p$. Since $\cW(\theta_0)=0$, $s_1>0$. By Theorem \ref{thm: HN for Q} we get $Y\in \overline\cQ(\theta_{s_1})$. Do this for every $Y\in \cY$ to get another $s>0$ so that $\cY\subset \overline\cQ(\theta_s)$. Let $s'=s/2$. Then $\overline \cQ(\theta_s)\subseteq \cQ(\theta_{s'})\subseteq \cQ(\theta_0)=\overline\cQ(\theta_0)$ we obtain $\cQ(\theta_{s'})=\overline\cQ(\theta_0)$. Thus $\cP(\theta_0)=\,^\perp \overline\cQ(\theta_0)=\,^\perp\cQ(\theta_{s'})=\overline\cP(\theta_{s'})$.

Let $w_1,\cdots,w_n$ be the family of vectors whose coordinates are all $1$ except the $i$th coordinate of $w_i$ which is 2. These are linearly independent. Let $\theta_t^{(i)}$ be the green paths given by $\theta_t^{(i)}=\theta_0+tw_i$. Repeating the process described in the last two paragraphs, we obtain $t_i<0$ and $s_i>0$ so that
\[
	\cP(\theta_{t_i}^{(i)})=\cP(\theta_0)=\overline\cP(\theta_{s_i}^{(i)}).
\]
By Lemma \ref{lem: path in U}, the points $\theta_{t_i}^{(i)}$ and $\theta_{s_i}^{(i)}$ lie in the pseudo-chamber $\cU$ containing $\theta_0$. But $\cU$ is convex and the convex hull of the points $\theta_{t_i}^{(i)}$ and $\theta_{s_i}^{(i)}$ contains $\theta_0$ in its interior. Therefore $\theta_0$ lies in the interior of $\cU$. But $\theta_0$ was an arbitrary point in the pseudo-chamber $\cU$ having functorially finite $\cP(\cU)$. Therefore, $\cU$ is open.
\end{proof}

In Remark \ref{rem: gamma gives generators and cogenerators}, we explained without proof how to compute generators and cogenerators for pseudo-torsion classes associated to open pseudo-chambers. The proof follows the ideas in the above theorems.
\ki{
\begin{thm}\label{thm: green paths gives generators and cogenerators}
    Let $\cU$ be a pseudo-chamber in the stability diagram of $\cG$. Let $\gamma$ be any green path so that $\gamma(0)\in\cU$. Then the pseudo-torsion class $\cP(\theta)$ is generated by all $\cW(\theta_t)$ for all $t<0$ and $\cQ(\theta)$ is generated by all $\cW(\theta_t)$ for $t>0$.
\end{thm}

\begin{proof}
    This follows from the strict HN-stratification of $\cG$ given by the path $\gamma$. Any object $M$ in $\cP(\theta_0)$ has a strict filtration with subquotients in $\cW(\theta_{t_i})$ for some $t_i<0$. Similarly, any object of $\cQ(\theta_0)$ has a strict filtration with subquotients in $\cW(\theta_{s_j})$ where $s_j>0$.
\end{proof}
}

In Example \ref{fig: eg, B2} we saw there were three open pseudo-chambers but the corresponding pseudo-torsion classes were either finitely generated or finitely cogenerated, but not both. We will generalize these examples.

\begin{defn}\label{def: red and green}
    Let $\cU$ be a pseudo-chamber with interior $int\,\cU$ and boundary $\partial \cU=\overline\cU-int\,\cU$. A boundary point $\theta\in \partial\cU$ is called \textbf{green} if $\theta-\varepsilon\eta\in int\,\cU$ for all sufficiently small $\varepsilon>0$. Here $\eta=(1,1,\cdots,1)$. Similarly $\theta\in\partial\cU$ is called a \textbf{red boundary point} of $\cU$ if $\theta+\varepsilon\eta\in int\,\cU$ for sufficiently small $\varepsilon>0$.
\end{defn}

\begin{cor}\label{cor: green boundary}
        Suppose that $\theta_0\in V_\Lambda$ with $\cW(\theta_0)=0$. Thus $\theta_0$ lies in a pseudo-chamber $\cU$. Suppose $\cP(\theta_0)$ is finitely generated. Then $\cU$ has a nonempty interior and $\theta_0$ lies either in the interior of $\cU$ or is a green boundary point of $\cU$.
\end{cor}

\begin{proof}
    This follows from the proof of Theorem \ref{thm: ff P,Q implies U open}. Let $w_1,\cdots,w_n$ be linearly independent positive vectors whose sum is $(n+1)\eta$ and let $\theta_t^{(i)}=\theta_0+tw_i$. Then, as in the proof of Theorem \ref{thm: ff P,Q implies U open}, there exist $t_i<0$ so that $\cP(\theta_{t_i}^{(i)})=\cP(\theta_0)$. By Lemma \ref{lem: path in U}, we have $\theta_{t_i}^{(i)}\in \cU$. Since $\cU$ is convex it contains the convex hull of these points together with $\theta_0$. This is a solid $n$-simplex in $V_\Lambda$ whose interior is an open set contained in $\cU$. But $\theta_0-\varepsilon\eta$ lies in the interior of this simplex for small $\varepsilon>0$. So, either $\theta_0$ lies in the interior of $\cU$ or it is a green boundary point of $\cU$.
\end{proof}

By an analogous argument we obtain the dual statement:

\begin{cor}\label{cor: red boundary}
    Suppose $\theta_0$ lies in a pseudo-chamber $\cU$ and $\cQ(\theta_0)$ is finitely generated. Then $\cU$ has nonempty interior and $\theta_0$ lies either in the interior of $\cU$ or is a red boundary point of $\cU$.
\end{cor}
}

\ki{\begin{eg}\label{eg: preinjectives}
Let $\cG$ be the torsion class of preinjective modules over the Kronecker algebra $KQ$ where $Q$ is the quiver
\begin{center}
\begin{tikzpicture}
\draw (-.2,0) node{$1$};
\draw (1.2,0) node{$2$};
\draw[thick,<-] (0,0.1)..controls (.2,.2) and (.8,.2)..(1,0.1);
\draw[thick,<-] (0,-0.1)..controls (.2,-.2) and (.8,-.2)..(1,-0.1);
\end{tikzpicture}
\end{center}
Then the $\Lambda$-indecomposable objects of $\cG$ are the preinjective modules with dimension vectors $(k+1,k)$ where $k\ge0$. The point $\theta_0=(-1,1)$ has $\cP(-1,1)=0$ since, in order for $\theta_0(M)>0$ we need the second coordinate of $\undim M$ to be greater than the first. However, $\cP(\varepsilon-1,1)$ is nonzero. Also, $\cP(1,-1)=\cG$ since $(1,-1)\cdot \undim M>0$ for all nonzero objects $M$ in $\cG$. So, $\theta_1=(1,-1)$ is a red boundary point of the pseudo-chamber $\cU_2$ with $\cP(\cU_2)=\cG$. See Figure \ref{fig: eg, Kronecker}.  
\end{eg}}

{
\begin{figure}[htbp]
\begin{center}
\begin{tikzpicture}[yscale=1.5]
%
\draw[thick] (0,1.2)--(0,-1.2); 
\draw (0,1.2) node[above]{$D_\cG(I_1)$};
\draw (0,-1.2) node[below]{$D_\cG(I_1)$};

\coordinate (L0) at (-1,-.5);
\coordinate (L1) at (.4,-.8);
\coordinate (L2) at (1.6,-1.1);
\coordinate (L3) at (1,.5);
\coordinate (L5) at (-1.5,1.1);
\coordinate (L4) at (-.5,1.1);
\draw (L0)node{$\cU_0$};
\draw (L1)node{$\cU_1$};
\draw (L3)node{$\cU_2$};
\draw (-.9,.8)--(.9,-.8);
\draw (-1,.8)--(1,-.8);
\draw (-1.1,.8)--(1.1,-.8);
\draw (-1.2,.8)--(1.2,-.8);
\draw (-1.3,.8)--(1.3,-.8);
\draw (-1.4,.8)--(1.4,-.8);
\draw[very thick,green!80!black] (0,0)--(-1.5,.8) (-.75,.4)node{$\bullet$} (-1.1,.4)node[below]{\tiny$(-1,1)$};
\draw[very thick,red] (0,0)--(1.5,-.8) (.75,-.4)node{$\bullet$} (1.1,-.4)node[above]{\tiny$(1,-1)$};

\draw (-.8,.8)--(0,0);
\draw (-.7,.8)--(0,0);
\draw (-.6,.8)--(0,0);
\draw (-.5,.8)--(0,0);
\draw (-.4,.8)--(0,0);
\draw (-.3,.8)--(0,0);
\draw (-.2,.8)--(0,0);
\draw (-.1,.8)--(0,0);
\end{tikzpicture}
\caption{The strict wall-and-chamber structure for $\cG$, the preprojective modules over the Kronecker quiver is shown. There is one open pseudo-chamber $\cU_1$ and two other pseudo-chambers with nonempty interiors $\cU_0$ and $\cU_2$. With corresponding pseudo-torsion classes $\cP_0$, $\cP_1$ and $\cP_2$. 
The pseudo-torsion classes $\cP_0=0$ and $\cP_1=Filt\,I_1$ are finitely generated and $\cP_2=\cG=\,^\perp 0$ is finitely cogenerated. So, the pseudo-chamber $\cU_0$ contains some of its green boundary points and $\cU_2$ contains some of its red boundary points. For example $(-1,1)\in \cU_0$ is a green boundary point and $(1,-1)\in\cU_2$ is a red boundary point.}
\label{fig: eg, Kronecker}
\end{center}
\end{figure}
}

\section{Resolution of thick walls}\label{new sec4}

This is a preview of our next paper \cite{IM2} in which we discuss maximal green sequences given by green paths and resolve the problem that arises when the green path goes through thick walls. 

The poset of pseudo-torsion classes in a torsion class $\cG$ forms a lattice since any intersection of pseudo-torsion classes is a pseudo-torsion class and any collection of pseudo-torsion classes generates a pseudo-torsion class. A \textbf{maximal green sequence} in $\cG$ is defined to be a maximal linearly ordered tower in this lattice. Given a green path $\theta_t$ which only goes through thin walls $D_\cG(B_\beta)$ at times $t_\beta$, we obtain a FHO-sequence $\{B_\beta\}$ with $\pi(B_\beta)=t_\beta$. This is a possibly countably infinite collection of pseudo-bricks. This gives a possibly continuum number of pseudo-torsion classes $\cP(\theta_\delta)$ where $\delta\in\RR$ ranges through the Dedekind cuts in the countable family of real numbers $t_\beta$. This includes both the components of the complement of the set of $t_\beta$ and the infinitesimal green and red side of each pseudo-wall, namely $\cP(\theta_{t_\beta})$ and $\overline\cP(\theta_{t_\beta})$.

When a green path goes through a thick wall, such as the null domain for a tame hereditary algebra, it cannot distinguish between the objects in $\cW(\theta)$ since they all have the same (or proportional) dimension vectors. To separate these objects we need a new stability space which may have uncountable dimension. We will assume the cardinality of the underlying field $K$ is at most that of $\RR$. Then the set of isomorphism classes of finitely generated $\Lambda$-modules will have cardinality at most the continuum. Then we can find smooth green paths which pass through a possibly continuum number of walls at different times.

\subsection{Stability space of a wide subcategory}\label{ss41: VX}

We take the pseudo-wide subcategory $\cX=\cW(\theta)$ for $\theta$ on a thick wall with strict morphisms. The first step is to discuss $K_0\cX$.

\begin{defn}\label{def: K0X}
    For any pseudo-wide subcategory $\cX$, we define $\overline K_0\cX$, the \textbf{reduced K-theory of $\cX$} to be the additive group with generators the formal expression $[X]$ where $X\in \cX$ modulo the relation that $[A]-[B]+[C]=0$ for all strict exact sequences $A\cof B\onto C$ and the additional relation that $[nX]=n[X]$ where $nX$ is the direct sum of $n$ copies of $X\in\cX$. (This is to take care of pathological examples such as Example \ref{B2 example} where $nI_2$ is strictly indecomposable for all $n\ge1$.) The \textbf{reduced stability space} of $\cX$ is the vector space
\[
    V_\cX:=\Hom(\overline K_0\cX,\RR).
\]
\end{defn}

Note that since the strict category $\cG$ is a subcategory of $mod\text-\Lambda$ we have a homomorphism $K_0\cG\to K_0\Lambda$ which factors through $\overline K_0\cG$ since the relation $[nX]=n[X]$ holds in $K_0\Lambda$.

For any $\widetilde\theta\in V_\cX$ and any object $X\in\cX$ we have
\[
    \widetilde\theta(X):=\widetilde\theta([X])\in\RR
\]
where $[X]$ is the element of $K_0(\cX)$ determined by $X$. By definition of $K_0\cX$, these functions satisfy the condition that $\widetilde\theta(B)=\widetilde\theta(A)+\widetilde\theta(C)$ for any strict extension $B$ of two objects $A,C\in\cX$.

Although the space $V_\cX$ may be infinite dimensional, we can still talk about smooth paths in $V_\cX$. Namely, a path $\widetilde\theta_t$, $t\in \RR$ or $t\in[0,1]$ is a path in $V_\cX$ so that $\widetilde\theta_t(X)$ is a continuous function of $t$ for all $X\in\cX$. The path is smooth if $\widetilde\theta_t(X)$ is a smooth function of $t$ for all $X\in\cX$.

We can construct relative pseudo-torsion classes $\cP(\widetilde\theta)$, relative pseudo-torsionfree classes $\cQ(\widetilde\theta)$ and relative pseudo-wide subcategories which, fortunately, are pseudo-wide subcategories, allowing us to go to the next level. These are defined as expected in the following way.

{%

\begin{defn}\label{def: P,Q,W for X}
For any $\widetilde\theta\in V_\cX$ we define $\cP(\widetilde\theta)$, $\overline\cP(\theta)$, $\cQ(\theta)$, $\overline\cQ(\theta)$, $\cW(\widetilde\theta)$ in a way analogous to Definition \ref{def: P,Q,W}. For example $\cP(\widetilde\theta)$ is the class consisting of $0$ and all nonzero objects $M$ in $\cX$ so that $\widetilde\theta(M)>0$ and $\widetilde\theta(M'')>0$ for all nonzero strict quotients $M''$ of $M$ in $\cX$. And $\cW(\widetilde\theta)$ is the category of all objects $X\in\cW$ so that $\widetilde\theta(X)=0$ and $\widetilde\theta(X'')\ge0$ for all strict quotients $X''$ of $X$ in $\cX$. Note that $0$ is an object in each of these classes. Also, $\cP(\widetilde\theta)\subseteq \cX$ by definition.
\end{defn}
}

{
We have the following theorem analogous to Theorem \ref{thm: P(theta) is P} with similar proof.
\begin{thm}\label{thm: P(tilde-theta) is P}
    For every $\widetilde\theta\in V_\cX$, $\cP(\widetilde\theta)$ is a relative pseudo-torsion class in $\cX$ inside $\cG$.
\end{thm}

}
{
Analogous to Theorems \ref{thm: P(theta) is constant on path components} and \ref{thm: P(theta) are different in different path components} we have the following.

\begin{thm}\label{thm: P(tildetheta) on chambers}
    $\cP(\widetilde\theta)$ is constant on the pseudo-chambers of $V_\cX$ and distinct on disjoint pseudo-chambers.
\end{thm}
}

{
Since the strict category $\cG$ is a pseudo-wide subcategory of itself, the above theory applies to the stability space
\[
    V_\cG=\Hom(\overline K_0\cG,\RR)
\]
Since $\cG\subset mod\text-\Lambda$, we have an additive map $\psi: \overline K_0\cG\to K_0\Lambda$ as discussed above. This map is usually surjective and thus induces an $\RR$-linear map $\psi^\ast: V_\Lambda\to V_\cG$ which is usually a monomorphism.}

{\begin{eg}\label{eg: three circles}
    Let $\Lambda$ be the Hereditary algebra of type $B_2$ given by the modulated quiver $\CC\to \RR$. Thus $S_1=I_1=(\CC\to 0)$ is simple injective, $S_2=(0\to \RR)$ is simple projective, $P_1=(\CC\to \RR^2)$ is projective. Let $\cG=\,^\perp S_2$ be the torsion class with only $P_1,I_1,I_2$ and extensions. This was explained in detail in Example \ref{B2 example} and further explained in Remark \ref{rem: not extension closed}. Here we look at the reduced stability diagram \ki{in $V_\cG$ of this strict category $\cG$. We see the open pseudo-chambers} holding the 4 torsion classes $\cG_0=0$, $\cG_1=add(I_1)$, $\cG_2=add(I_1\oplus I_2)$ and $\cG_3=\cG$. 
    
    \ki{The first two pseudo-torsion classes are finitely generated:  $\cG_0=Filt(0)$, $\cG_1=Filt(I_1)$. The last two are finitely cogenerated: $\cG_2=\,^\perp P_1$, $\cG=\,^\perp 0$. There is one more open pseudo-chamber with finitely generated pseudo-stability class $\cP_6=Filt\{P_1,I_1\}$. This one does not occur in Figure \ref{fig: eg, B2}.
    }
   
  \ki{Three reduced chambers which are shaded in Figure \ref{Fig: 3 circles} we label with the union of the pseudo-torsion classes in their interior.}
 The pseudo-torsion classes $\cP_4,\cP_5$ as well as $\cG_0,\cG_1,\cG_2,\cG_3$ appeared in the standard stability diagram \ref{B2 example} and the remaining pseudo-torsion classes $\cP_6$ and $\cP_7$ are discussed in Remark \ref{rem: not extension closed}.
    \end{eg}
    }
    
    
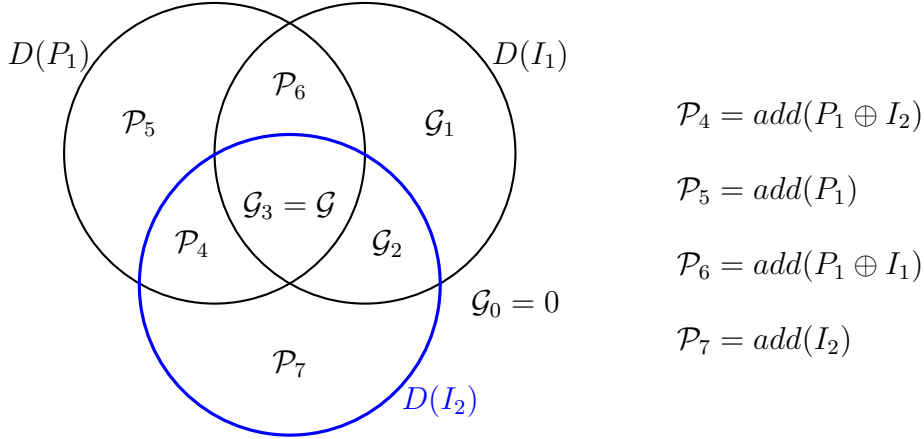
{   
\begin{figure}[htbp]
\begin{center}
\begin{tikzpicture}
\begin{scope}
       \draw[fill, gray!50!white] (-1,0) circle[radius=2cm]; 
         \draw[fill, gray!50!white] (0,-1.75) circle[radius=2cm]; 
       \draw[fill,white] (1,0) circle[radius=2cm]; 
\end{scope}

\begin{scope}[xshift=4cm,yshift=-15mm]
	\draw(1,2) node[right]{$\cP_4=\,^\perp I_1$};
	\draw(1,1) node[right]{$\cP_5=\cP_4\cap \cP_6=\,^\perp\{I_1,mI_2, m\ge1\}$};
	\draw(1,0) node[right]{$\cP_6=Filt\{P_1, I_1\}$};
	\draw(1,-1) node[right]{$\cP_7=\cP_2\cap \cP_4=\,^\perp\{P_1,I_1\}$};
\end{scope}
\coordinate (G0) at (3,-2);
\coordinate (G1) at (2,0.4);
\coordinate (G2) at (1.3,-1.2);
\coordinate (G3) at (0,-.7);
\coordinate (P4) at (-1.3,-1.2);
\coordinate (P5) at (-2,0.4);
\coordinate (P6) at (0,.9);
\coordinate (P7) at (0,-2.8);

        \draw[thick] (-1,0) circle[radius=2cm]; 
        \draw[thick] (1,0) circle[radius=2cm]; 
        \draw[very thick,blue] (0,-1.75) circle[radius=2cm]; 
        \draw (-3.2,1.3) node{$D(P_1)$};
        \draw (3.2,1.3) node{$D(I_1)$};
       \draw[blue] (2,-3.3) node{$D(I_2)$};
       
       \draw (G0) node{$\cG_0=0$};
       \draw (G3) node{$\cG_3=\cG$};
      \draw (G1) node{$\cG_1$};
      \draw (G2) node{$\cG_2$};
     \draw (P4) node{$\cP_4$};
     \draw (P5) node{$\cP_5$};
     \draw (P6) node{$\cP_6$};
     \draw (P7) node{$\cP_7$};

\draw (0,-1.75) node{$\bullet$} node[below]{$\theta_0$};

\end{tikzpicture}
\caption{\color{black}This is the reduced stability diagram for Example \ref{eg: three circles}. The three shaded regions are unions of infinitely many pseudo-chambers labeled with the union of their pseudo-torsion classes. For example $\cP_7=\cP(\theta_0)$ where $\theta_0=(0,1,0)$ in coordinates $[P_1],[I_2],[I_1]$. $\theta_0(mI_2\oplus kI_1)=m>0$ for all nonzero elements $mI_2\oplus kI_1$ in $\cP_7$.}
\label{Fig: 3 circles}
\end{center}
\end{figure}
}

{
\begin{defn}\label{def: green path in VX}
We define a \textbf{smooth green path} in $V_\cX$ to be a smooth path $\widetilde\theta_t, t\in\RR$ in $V_\cX$ satisfying the following where $D_\cX(X)$ is the set of all $\widetilde\theta\in V_\cX$ so that $X\in\cW(\widetilde\theta)$.
\begin{enumerate}
    \item whenever $\widetilde\theta_{t_0}$ lies in $D_\cX(X)$ of $X$, then
    \[
        \frac{d}{dt}\widetilde\theta_{t_0}(X)>0
    \]
    \item There is an $U\in\RR$ so that $\widetilde\theta_t(X)>0$ for all $t>U$ and all nonzero $X\in \cX$.
    \item There is an $L\in\RR$ so that $\widetilde\theta_t(X)<0$ for all $t<L$ and all nonzero $X\in \cX$. 
\end{enumerate}
\end{defn}
}

\subsection{Level 2 thick walls}\label{ss42: Level 2}

We need to arrange that the walls with pseudo-simple labels are quasi-thin. To do this, we choose a Hamel basis $\cH$. This is a basis for $\RR$ as vector space over $\QQ$. Hamel showed that the Hamel basis has the cardinality of the continuum \cite{Hamel}. See also \cite{Mackey}. By translation we may choose a Hamel basis in the unit interval $\cH\subset [0,1]$. Since there are at most a continuum number of pseudo-simple objects $S_\alpha$, indexed by a set $\cA=\{\alpha\}$, which give a basis for $\overline K_0\cG$, (not counting the indecomposable objects of the form $nX$ as in Example \ref{B2 example}), we can find a monomorphism $\lambda:\cA\into \cH$. Then, let $\widetilde\theta_t$ be the linear green path in $V_\cG$ given by $\theta_t(S_\alpha)=t-\lambda(\alpha)$. Then we claim that, for each $\alpha\in\cA$, $\cW(\widetilde\theta_{\lambda(\alpha})$ is a quasi-thin wall. The other walls have objects of dimension at least twice the minimum dimension of any $S_\alpha$. So, if we repeat the process the objects will have unbounded size and that will allow us to get, after a lot of effort, iterated refinements of any green path which will give a ``maximal green sequence'', a maximal chain in the lattice of quasi-torsion classes in $\cG$. Details will be in the next paper \cite{IM2}.

\section{Algebraic K-theory}\label{sec 5: Alg K-theory}

In this section we use Example \ref{eg: A3 example stability} to show how the strange new walls $D_\cG(I_2\oplus I_1)$ and $D_\cG(P_1\oplus I_1)$ can be interpreted as the ``ghosts'' of the deleted objects $S_2$ and $P_2$. We also show how these ghosts are used in the computation of the ``Grassmann invariant''
\[
\chi_c:K_3\ZZ\pi\to H_0(\ZZ\pi;\ZZ_2\pi)
\]

In \cite{GrInvRedrawn} it is shown that elements of $K_3\ZZ\pi$ are represented by ``pictures'' that represent stability diagrams of quivers of type $A_3$. We use Example \ref{eg: A3 example stability} to illustrate how $\chi_c$ can be computed on a diagram and how ``ghosts'' are used to show that $\chi_c=0$ on certain diagrams. See \cite{GrInvRedrawn} for further details. The topological view is also explained in \cite{KIsecond}. The presentation we give here is slightly different from that in \cite{GrInvRedrawn} and \cite{MoreGhosts} and we believe that we now have better formulas for the ``ghosts'' using the strict category $\cG$.

We relabel the modules for $A_n$ with straight orientation
\[
    1\to 2\to \cdots\to n+1
\]
using one of the standard notations: $M_{ij}=P_i/P_j$  where $P_i$ is the projective cover of $S_i$ and $P_{n+1}=0$. In this notation, the $\Lambda$-indecomposable objects in the torsion class $\cG$ in Example \ref{eg: A3 example stability} are circled in the Auslander-Reiten quiver:
\begin{center}
\begin{tikzpicture}[scale=1.4]
{ 
\begin{scope}[xshift=4cm,yshift=-.5cm]
\draw (0,0) node{$M_{34}$}; 
\draw(0,0)circle[radius=2.7mm];
\draw (0.5,0.6) node{$M_{24}$}; 
\draw (1.5,0.6) node{$M_{13}$};
\draw(1.5,0.6)circle[radius=2.7mm];
\draw (1,1.2) node{$M_{14}$};
\draw(1,1.2)circle[radius=2.7mm];
\draw (1,0) node{$M_{23}$}; 
\draw (2,0) node{$M_{12}$}; 
\draw(2,0)circle[radius=2.7mm];
\draw[->] (0.13,0.15)--(.33,.4);
\draw[->] (1.13,0.15)--(1.33,0.4);
\draw[->] (0.63,0.75)--(.83,1);
\draw[->] (0.63,0.37)--(.83,.14);
\draw[->] (1.63,.37)--(1.83,.14);
\draw[->] (1.13,.97)--(1.33,.74);
\end{scope}
}
\end{tikzpicture}
\end{center}

{
In Figure \ref{Fig: picture for K3}, on the left, we have the stability diagram of $\cG$ with two of the new walls $D_\cG(Z_{123})$ and $D_\cG(Z_{124})$ inserted where $Z_{123}:=M_{13}\oplus M_{12}$ and $Z_{124}:=M_{14}\oplus M_{12}$. These are indecomposable in the strict category $\cG$ by Proposition \ref{prop: when A+B is G-indecomposable} since we have non-admissible morphisms $M_{13}\to M_{12}$ with kernel $M_{23}\notin \cG$ and $M_{14}\to M_{12}$ with kernel $M_{24}\notin \cG$. We interpret $Z_{123}$ and $Z_{124}$ as the ``ghosts'' of the deleted objects $M_{23}$ $M_{24}$.

At the points $a,b,c,d,e$ in Figure \ref{Fig: picture for K3} on the left we have strict extensions:
\begin{enumerate}
\item[(a)] $M_{12}\cof M_{13}\oplus M_{12} \onto M_{13}$
\item[(b)] $M_{12}\cof M_{14}\oplus M_{12} \onto M_{14}$
\item[(c)] $M_{12}\cof M_{13}\oplus M_{12} \onto M_{13}$
\item[(d)] $M_{34}\cof M_{14}\oplus M_{12} \onto M_{13}\oplus M_{12}$
\item[(e)] $M_{13}\cof M_{14} \onto M_{34}$
\end{enumerate}
}

In Figure \ref{Fig: picture for K3}, on the right, we reinterpret Figure \ref{Fig: A3 example} as the bifurcation diagram of a 2-parameter family of Morse functions on an $h$-cobordism $W$ with $\pi_1W=\pi$. The labels $e_{ij}^u$ indicate column operations on the incidence matrices $r,r',r''$ for the cellular chain complex of $W$ with coefficients in $\ZZ\pi_1W$. (See Figure \ref{Fig: handle slides}.) Since $W$ is an $h$-cobordism, these matrices are invertible and we denote their inverses by $s,s',s''$, Thus we have
\begin{enumerate}
    \item $r'=re_{34}^v$, $s'=e_{34}^{-v}s$
    \item $r''=re_{34}^ve_{14}^{uv}$, $s''=e_{14}^{-uv}e_{34}^{-v}s$
\end{enumerate}

The formula for $\chi_c:K_3\ZZ\pi\to H_0(\ZZ\pi;\ZZ_2\pi)
$ is explained geometrically in \cite{GrInvRedrawn} and algebraically in \cite{A-infty}. Elements of $K_3\ZZ\pi$ are represented by pictures such as the one on the right side of Figure \ref{Fig: picture for K3} (but without the blue and red arcs labeled $z_{123}^{w,u}$ and $z_{124}^{w,uv}$).

The value of $\chi_c$ on a picture is given by adding up contributions at vertices at which we have a commutator relation $r[e_{ij}^u,e_{ik}^v]$ (which means $re_{ij}^ue_{ik}^v=re_{ik}^ve_{ij}^u$) such as at vertices $a,b,c$. The formula for the contribution to $\chi_c$ given in \cite{GrInvRedrawn} is
\[
    z_{ijk}^{u,v}(r):=\chi_c\left(r[e_{ij}^u,e_{ik}^v]\right)=\sum_p r_{pi}\left<us_{jp},vs_{kp}\right>\in \ZZ_2\pi
\]
where for any $x,y\in \ZZ\pi$ we define $\left< x,y\right>$ to be the intersection of $\overline x,\overline y\in\ZZ_2\pi$ where $\overline x,\overline y$ are the images of $x,y$ in $\ZZ_2\pi$ interpreted as the set of finite subsets of $\pi$. Thus, the total value of $\chi_c$ on the entire picture is given by the sum: $\chi_c(a)+\chi_c(b)+\chi_c(c)$ which vanishes by the following relations which can be easily verified using the relation $\left<x,x\right>=\overline x$ and the bilinearity of $\left<x,y\right>$.
\begin{enumerate}
\item At $f$ we see the relation $z_{123}^{w,u}(r')=z_{123}^{w,u}(r'')$.
\item At $d$ we see the relation $z_{123}^{w,u}(r)=z_{123}^{w,u}(r')+z_{124}^{w,uv}(r')$.
\end{enumerate}
Also, $z_{123}^{w,u}(r)=z_{123}^{w,u}(re_{12}^w)$ and relations (1) and (2) hold with $r,r',r''$ replaced with $re_{12}^w,r'e_{12}^w,r''e_{12}^w$. 

Thus the ghost walls $D_\cG(Z_{123})$ and $D_\cG(Z_{124})$ on the stability diagram of $\cG$ indicate a proof of how the terms in $\chi_c$ can be canceled to give 0. This is important since the right side of Figure \ref{Fig: picture for K3} gives the zero element of $K_3\ZZ\pi$. So, this calculation is part of the proof that $\chi_c$ is well defined on $K_3\ZZ\pi$. See \cite{GrInvRedrawn} for details. 

\begin{figure}[H]
\begin{center}
{ 
\begin{tikzpicture}

{\begin{scope}[xshift=-8cm]
\coordinate (A) at (-1.2,-2);
\coordinate (B) at (1.95,-.5);
\coordinate (C) at (0,1.73);
\draw[ thick] (-1,0) circle[radius=2cm];
\draw[ thick] (0,-2) ellipse[x radius=3cm,y radius=2cm];
\draw[ thick] (A)..controls (1,-3.5) and (2,-2)..(B)..controls (2,0.7) and (1,1.7)..(C);
\coordinate (G0) at (-3,-3.5);
\coordinate (G1) at (-1.5,-3);

\coordinate (G2) at (-1.7,.9);
\coordinate (G3) at (-2,-1);

\coordinate (G6) at (-.3,-.7);

\coordinate (G9) at (0,-.6);
\coordinate (G4) at (.5,-2.9);
\coordinate (G5) at (0.2,-2.2);

\coordinate (D1b) at (2.5,1.3);
\draw (D1b) node{$\tiny D(M_{13})$};

\coordinate (D1) at (2.8,-3.5);
\draw (D1) node{$\tiny D(M_{12})$};
\coordinate (D2) at (-3,1.5);
\draw (D2) node{$\tiny D(M_{34})$};

\coordinate (D3) at (0,-3.5);

\coordinate (D4) at (1.2,-1.7);
\draw (D4) node{$\tiny D(M_{14})$};

\draw[ thick] (.65,-.8) ellipse[x radius=2.04cm, y radius=2.7cm];
\draw[thick, green!80!black,<-] (-1,-1.5)--(-3.5,-1.5) node[left]{$\gamma$};
\draw[blue](-1.33,-.22) node{$\bullet$};
\draw[blue](-1.18,-.2) node[below]{$a$};

\draw[blue,very thick](-1.33,-.22).. controls (-.5,1.3) and (2,1.7)..(2.68,-1.1);
\draw (.75,.95)node{$\bullet$};
\draw (.6,.95)node[below]{$d$};
\draw[very thick,red] (.75,.95)..controls (1.5,.5) and (1.7,0)..(1.95,-.5);

\draw[blue] (0.6,1.1) node[left]{\tiny$D(Z_{123})$};
\draw[red] (1.9,.2) node[left]{\tiny$D(Z_{124})$};

\draw (1.95,-.5) node{$\bullet$};
\draw (2,-.6) node[left]{$b$};
\draw (2.68,-1.1) node{$\bullet$}node[left]{$c$};
\coordinate (E) at (-1.18,-2);
\draw (E) node{$\bullet$};
\draw (-1.3,-2)node[below]{$e$};
\end{scope}
}
{\begin{scope}
\coordinate(F) at (1.7,.58);
\draw (F) node{$\bullet$};
\draw(1.6,.7) node[right]{$f$};

\coordinate (E) at (-1.18,-2);
\draw (E) node{$\bullet$};
\draw (-1.3,-2)node[below]{$e$};

\coordinate (A) at (-1.2,-2);
\coordinate (B) at (1.95,-.5);
\coordinate (C) at (0,1.73);
\draw[ thick] (-1,0) circle[radius=2cm];
\draw[ thick] (0,-2) ellipse[x radius=3cm,y radius=2cm];
\draw[ thick] (A)..controls (1,-3.5) and (2,-2)..(B)..controls (2,0.7) and (1,1.7)..(C);
\coordinate (G0) at (-3,-3.5);
\coordinate (G1) at (-1.5,-3);

\coordinate (G2) at (-1.7,.9);
\coordinate (G3) at (-2,-1);

\coordinate (G6) at (-.3,-.7);
\draw (-.3,-1) node{$r$};
\draw (1.3,-1) node{$r'$};
\draw (2.2,-1.6) node{$r''$};

\coordinate (G9) at (0,-.6);
\coordinate (G4) at (.5,-2.9);
\coordinate (G5) at (0.2,-2.2);

\coordinate (D1b) at (2.4,1.3);
\draw (D1b) node{$\tiny e_{13}^u$};

\coordinate (D1) at (2.6,-3.5);
\draw (D1) node{$\tiny e_{12}^w$};
\coordinate (D2) at (-2.8,1.5);
\draw (D2) node{$\tiny e_{34}^v$};

\coordinate (D3) at (0,-3.5);

\coordinate (D4) at (1.3,-1.7);
\draw (D4) node{$\tiny e_{14}^{uv}$};

\draw[ thick] (.65,-.8) ellipse[x radius=2.04cm, y radius=2.7cm];
\draw[thick, green!80!black,<-] (-1,-1.5)--(-3.5,-1.5) node[left]{$\gamma$};
\draw[blue](-1.33,-.22) node{$\bullet$};
\draw[blue](-1.18,-.2) node[below]{$a$};

\draw[blue,very thick](-1.33,-.22).. controls (-.5,1.3) and (2,1.7)..(2.68,-1.1);
\draw (.75,.95)node{$\bullet$};
\draw (.6,.95)node[below]{$d$};
\draw[very thick,red] (.75,.95)..controls (1.5,.5) and (1.7,0)..(1.95,-.5);

\draw[blue] (0.3,1.1) node[left]{\tiny$z_{123}^{w,u}$};
\draw[red] (1.9,.2) node[left]{\tiny$z_{124}^{w,uv}$};

\draw (1.95,-.5) node{$\bullet$};
\draw (2,-.6) node[left]{$b$};
\draw (2.68,-1.1) node{$\bullet$}node[left]{$c$};

\end{scope}
} \draw (0,-5) node{$\ $};
\end{tikzpicture}
}

\caption{On the left is Figure \ref{Fig: A3 example} with modules relabelled $M_{ij}$ and two of the new walls $\color{blue}D_\cG(Z_{123})$ and $\color{red}D_\cG(Z_{124})$ added where $Z_{ijk}:=M_{ij}\oplus M_{ik}$. The right side shows a 2-parameter family of Morse functions on an $h$-cobordism $W$. The black arcs are the loci of handle slides. The red and blue arcs labelled $z_{ijk}^{u,v}$ are ``ghost handle slides''. These are used to compute the ``Grassmann invariant'' $\chi_c$.  At the points $a,b,c$ we generate three terms in the formula for $\chi_c$. These cancel at the point $d$ as explained in the text. The point $e$ shows the Steinberg relation $[e_{13}^u,e_{34}^v]=e_{14}^{uv}$. The path $\gamma$ gives the 1-parameter family of functions shown in Figure \ref{Fig: handle slides}.}
\label{Fig: picture for K3}
\end{center}
\end{figure}
\begin{figure}[H]
\begin{center}
{ 
\begin{tikzpicture}
\coordinate(B) at (-8,.9);
\coordinate(A) at (-8,1);
\coordinate(A1) at (-8,1);
\coordinate(A2) at (-8,1.1);
\coordinate(A3) at (-8.1,1.2);
\coordinate(A4) at (-8.1,1.3);
\draw[thick] (-5,0)--(5,0);
\draw[thick] (-5,4)--(5,4);
\draw[thick] (-5,3)--(5,3);
\draw[thick] (-5,2)--(5,2);
\draw[thick] (-5,1)--(5,1);
\draw (A) node[left]{$\ast$};
\draw[thick] (-5,-1.7)--(5,-1.7);
\draw[thick] (-5,-2.1)--(5,-2.1);
\draw[thick] (-5,-2.2)--(5,-2.2);
\draw[thick] (-5,-2.3)--(5,-2.3);
\draw[thick] (-5,-2.4)--(5,-2.4);
\draw[thick] (-5,-2.5)--(5,-2.5);
\draw (-5,-1.7) node[left]{$y_p$};
\draw (-5,-2.3) node[left]{$y_{p'}$};
\draw (-5,0) node[left]{$x_q$};
\draw (-5,1) node[left]{$x_1$};
\draw (-5,2) node[left]{$x_2$};
\draw (-5,3) node[left]{$x_3$};
\draw (-5,4) node[left]{$x_4$};
\draw[->] (-3.5,3)--(-3.5,3.9);
\draw (-3.5,3.5) node[right]{$e_{34}^{-v}$};
\draw[->] (-2,1)--(-2,2.9);
\draw (-2,2.5) node[right]{$e_{13}^{-u}$};
\draw[->] (-.5,1)--(-.5,1.9);
\draw (-.5,1.5) node[right]{$e_{12}^{-w}$};
\draw[->] (1,3)--(1,3.9);
\draw (1,3.5) node[right]{$e_{34}^{v}$};
\draw[->] (2.5,1)--(2.5,3.9);
\draw (2.5,2.5) node[right]{$e_{14}^{uv}$};
\draw (.2,-.9) node{$r$};
\draw (1.7,-.9) node{$r'$};
\draw (3.2,-.9) node{$r''$};
\draw (-1.3,-.9) node{$re_{13}^u$};
\draw (-2.8,-.9) node{$re_{13}^ue_{12}^w$};
\draw (A4)..controls (-7,3) and (-6,4).. (-5.5,4);
\draw (A3)..controls (-7,2) and (-6,3).. (-5.5,3);
\draw (A2)..controls (-7,1) and (-6,2).. (-5.5,2);
\draw (A1)..controls (-7,0) and (-6,1).. (-5.5,1);
\draw (A)..controls (-7,-1) and (-6,0).. (-5.5,0);
\draw (B)..controls (-7,-1.5) and (-6,-1.7).. (-5.5,-1.7);
\begin{scope}[xshift=-6cm,yshift=-.9cm]
\draw(0,0) node{$r_{pq}$};
\draw[->] (.3,-.3)..controls (.5,-.3) and (.5,.3)..(.3,.3);
\end{scope}
\begin{scope}[xshift=-6cm,yshift=1.5cm]
\draw(0,0) node{$w$};
\draw[->] (.3,-.3)..controls (.5,-.3) and (.5,.3)..(.3,.3);
\end{scope}
\begin{scope}[xshift=-6cm,yshift=3.5cm]
\draw(0,0) node{$v$};
\draw[->] (.3,-.3)..controls (.5,-.3) and (.5,.3)..(.3,.3);
\end{scope}
\begin{scope}[xshift=-6cm,yshift=2.5cm]
\draw(0,0) node{$w^{-1}u$};
\draw[->] (.4,-.3)..controls (.6,-.3) and (.6,.3)..(.4,.3);
\end{scope}
%
\end{tikzpicture}
}
\caption{This represents the 1-parameter of Morse functions $f_t$ on an $h$-cobordism $W$ with $n$ critical points $x_i$ of one index and $n$ critical points $y_j$ of one lower index. Paths from the base point $\ast\in W$ to each critical point of $f_t$ are shown. Handle slides are shown as vertical arrows in the figure. An element of $\pi=\pi_1W$ is associated to each handle slide as shown. The incidence matrix for $f_t$ is an invertible matrix $r$ which changes to $re_{ij}^u$ when the parameter $t$ passes a handle slide $e_{ij}^u$. So, $r'=re_{34}^v$ and $r''=re_{34}^ve_{14}^{uv}$.}
\label{Fig: handle slides}
\end{center}
\end{figure}
%



\begin{thebibliography}{aa}

\bibitem[BCZ]{BCZ}Barnard, Emily, Andrew Carroll, and Shijie Zhu. "Minimal inclusions of torsion classes." Algebraic Combinatorics 2.5 (2019): 879-901.

\bibitem[BR]{Apostolos-Idun} Beligiannis, Apostolos, and Idun Reiten. \emph{Homological and homotopical aspects of torsion theories}. American Mathematical Soc., 2007.

\bibitem[Br]{Bridgeland} Bridgeland, Tom. \emph{Stability conditions on triangulated categories.} Annals of Mathematics (2007): 317--345.


\bibitem[BST]{BST}Br\"ustle, Thomas, David Smith, and Hipolito Treffinger. \emph{Wall and chamber structure for finite-dimensional algebras.} Advances in Mathematics 354 (2019): 106746.




\bibitem[H]{Hamel} Hamel, Georg. "Eine Basis aller Zahlen und die unstetigen L\"osungen der Funktionalgleichung: $f (x+ y)= f (x)+ f (y)$." Mathematische Annalen 60.3 (1905): 459--462.





\ki{\bibitem[HW]{HW} A. Hatcher and J. Wagoner, \emph{Pseudo-isotopies of compact manifolds}. Asterisque 6, Soci\'et\'e Math\'ematique de France, 1973.

\bibitem[Ha]{H2} A. Hatcher, \emph{The second obstruction for pseudo-isotopies}, Asterisque 6, Soci\'et\'e Math\'ematique de France, 1973. 239--275.

\bibitem[I1]{A-infty} K. Igusa,, \emph{On the algebraic K-theory of $A_\infty$-ring spaces}. Lecture Notes in Math 967: Algebraic K-theory. Springer, Berlin, Heidelberg, 1982. 146--194.

\bibitem[I2]{KIsecond} \bysame, \emph{Second obstruction to pseudoisotopy I}, arXiv:2110.09659, 2021.
}

\bibitem[I3]{Linearity} \bysame, \emph{Linearity of stability conditions}, Communications in Algebra (2020): 1--26.


\bibitem[I4]{GrInvRedrawn} \bysame. \emph{Generalized Grassmann invariant-redrawn}, arXiv:2502.19147.


\bibitem[I5]{MoreGhosts} \bysame. \emph{More ghost modules I}, arXiv:2506.12904.

\bibitem[IM]{IM2} Kiyoshi Igusa and Ray Maresca, \emph{Pseudo-torsion classes II}, in preparation.


\bibitem[IT]{IT14}
 Kiyoshi Igusa and Gordana Todorov, \emph{Picture groups and maximal green sequences},  Electronic Research Archive (2021), 29(5): 3031--3068 , doi: 10.3934/era.2021025.


\bibitem[IT3]{IT26}
\bysame. \emph{Balanced notation for $\tau$-rigid pairs}, Axioms, 15(10), 731. https://doi.org/10.3390/axioms15100731.


 




 \bibitem[Ki]{King} A.D. King: Moduli of representations of finite-dimensional algebras. Q. J. Math. Oxf. II. Ser. 45, 515--530 (1994).
 


\bibitem[Ma]{Mackey}Mackey, George W. \emph{On infinite dimensional linear spaces.} Proceedings of the National Academy of Sciences 29.7 (1943): 216--221.
 

\bibitem[P]{Peiffer} Ren\'ee Peiffer. \emph{\"Uber Identit\"aten zwischen Relationen}. Mathematische Annalen 121 (1949/1950): 67--99.



\bibitem[S]{Sierpinski} W. Sierpi\'nski, "Un th\'eor\`eme sur les ensembles continus", T\^ohoku Mathematical Journal, Vol. 13, pp. 300--303 (1918).




\end{thebibliography}
\end{document}